\documentclass[12pt]{amsart}
\usepackage[english]{babel}
\usepackage[utf8]{inputenc}
\usepackage[T1]{fontenc}
\usepackage{hyperref}
\usepackage{dsfont}
\usepackage{color}

\usepackage{stmaryrd}
\usepackage{cite}

\usepackage{pgf,tikz,pgfplots}
\pgfplotsset{compat=1.14}
\usepackage{mathrsfs}
\usetikzlibrary{arrows}

\usepackage{geometry}
\allowdisplaybreaks

\usepackage{amsmath}
\usepackage{amssymb}
\usepackage{amsthm}
\usepackage{amsfonts}

\newcommand{\An}{\mathrm{A}_n}
\renewcommand{\And}{\mathbf{A}_n^{\frac{1}{2}}}

\newcommand{\R}{\mathbb{R}}
\newcommand{\N}{\mathbb{N}}

\renewcommand{\div}{\text{div }}
\newcommand{\curl}{\text{curl }}

\newcommand{\B}{\mathcal{B}}
\renewcommand{\S}{\mathcal{S}}

\newcommand{\ty}{\tilde{y}}
\newcommand{\tv}{\tilde{v}}
\newcommand{\tyl}{\tilde{y}_l}
\newcommand{\tyr}{\tilde{y}_r}
\newcommand{\hy}{\hat{y}}
\newcommand{\hv}{\hat{v}}
\newcommand{\tval}{\tilde{v}^{\mathrm{aux},l}}
\newcommand{\tvar}{\tilde{v}^{\mathrm{aux},r}}

\newcommand{\tvaln}{\tilde{v}^{\mathrm{aux},l}_n}
\newcommand{\tvarn}{\tilde{v}^{\mathrm{aux},r}_n}

\theoremstyle{plain}
\newtheorem{thm}{Theorem}[section]
\newtheorem{lem}[thm]{Lemma}
\newtheorem{prop}[thm]{Proposition}
\newtheorem{cor}[thm]{Corollary}

\theoremstyle{definition}
\newtheorem{defn}{Definition}[section]

\theoremstyle{remark}
\newtheorem{rem}[thm]{Remark}

\title[Uniqueness for C-H equation with non-homogeneous BC]{Uniqueness for the Camassa-Holm equation with non-homogeneous boundary conditions.}
 \author{Florent Noisette}
  \date{\today}

\begin{document}
\maketitle

\tableofcontents

\begin{abstract}
We establish the uniqueness of solutions of the Camassa-Holm equation on a finite interval with non-homogeneous boundary conditions in the case of bounded momentum.
A similar result for the higher-order Camassa-Holm system is also given.
Our proofs rely on energy-type methods, with some multipliers given as solutions of some auxiliary elliptic systems.
\end{abstract}


\section{Introduction}

\subsection{Presentation of the models}

The \textit{Camassa-Holm equation} was first introduced by Fokas and Fuchsmeister in \cite{Fokas-Fuchs} for its similarity with the \textit{KdV equation}.
It was later re-derived by Camassa and Holm in \cite{C-H-ini} as a model for water waves in the shallow-water asymptotic under the influence of gravity and no surface tension.
It reads as follows 
\begin{equation} \label{e:C-H-stand}
\partial_t v - \partial_{txx}^3v
    + 2\kappa\,\partial_xv  + 3v\,\partial_x v
    = 2\partial_x v \,\partial_{xx}^2 v + v\,\partial_{xxx}^3v.
\end{equation}
We refer to \cite{ModShaWatLann,Constantin-Lannes,Bhatt-Mikh} for a discussion on the physical relevance of this equation in the context of water waves.
On the other hand, Camassa-Holm equations as well as its higher-order generalizations $(CH_n)$ are useful to describe geodesic flow for the Sobolev $H^n$ metric, see \cite{Constantin-Kolev,Constantin-Ivanov}.
The \textit{higher-order Camassa-Holm system} is introduce for $n\geq 1$ integer as
\begin{equation} \label{e:higher-order-CH}
\partial_t v = B_n(v,v),
\end{equation}
where $v$ is the unknown and $B_n$ is defined through
\begin{equation} \label{e:higher-order-CH-ope}
B_n(u,v) := - \An^{-1}(2\partial_x v \An(u)+v \An(\partial_x u)), 
\end{equation}
where the operator $A_n$ is
\begin{equation} \label{e:higher-order-CH-An}
\An      := \sum_{k=0}^{n}{(-\partial_x^2)^k}.
\end{equation}
with suitable boundary conditions.
The case $n=1$ corresponds to the Camassa-Holm equation.

Moreover, the Camassa-Holm equation was studied a lot because of some interesting features it displays:
it is \textit{bi-hamiltonian completely integrable}, in the sense that it admits a Lax pair,
which allows to construct infinitely many conservations laws, see \cite{Constantin-Escher-inv-scatt,Fokas-Fuchs};
in the case $\kappa=0$, it admits \textit{solitons solutions}, which do not evolve in $C^1$ 
-as they are peaked at the crest- 
and are referred to as \textit{peakons}, see \cite{Constantin-Escher-glob-blow,Lenells};
it also admits \textit{wave-breaking solution}, that is solutions whose $x$ derivative gets unbounded in finite time, see \cite{Constantin-Escher-wave-break,Constantin-Escher-blow-rate}.

The Cauchy problem for \eqref{e:C-H-stand} was extensively studied both on the torus and on the full line, see for example \cite{Coc-Hol-karl,Constantin-Escher-glob-blow,DKT-CH}. 
The initial and boundary value problem on a half-line as well as the one on a segment was also studied.
Escher and Yin extended the well posedness result in $H^s$ for $s>\frac{3}{2}$ and homogeneous boundary condition, see \cite{Esch-Yin-IBVP1, Esch-Yin-IBVP2}.
Zhang, Liu, and Qiao tackle the case of inhomogeneous boundary conditions on the whole line, see \cite{Cusp-CH}.
Then, Perrolaz proved weak-strong uniqueness in the case of inhomogeneous boundary condition and regularity  $W^{1,\infty}$, see \cite{Perrollaz}. 
We also mention that his result was later generalized for other similar equations, see \cite{twocompCH}.
We refer to \cite{Coc-Hol-karl2} for a study of the Cauchy problem for the higher-order Camassa-Holm system on a circle.

The aim of this article is to improve the weak-strong uniqueness for inhomogeneous boundary data stated in \cite{Perrollaz} into a stability estimate with regards to initial and boundary data,
what entails in particular a global in-time uniqueness result.
Our proof differs quite a lot from the previous ones as we do not use any characteristics to obtain our estimates, but rather an inequality on the relative energy between two solutions.
This method to derive estimates was initially used for the 2D incompressible Euler equation with non-homogeneous boundary conditions and bounded vorticity, first in \cite{pap-weig:ini}, then in \cite{Nois1} by the author.

We also deal with the higher-order Camassa-Holm equation, which has the same interesting feature as the Camassa-Holm equation to be recast as a transport-elliptic system.
Yet the order of the elliptic part of this system is precisely $2n$, where $n\geq 1$ is the integer introduced in \eqref{e:higher-order-CH}-\eqref{e:higher-order-CH-An}, and the energy methods have to be carried out differently.

\subsection{Definitions and statement of the main result for the Camassa-Holm equation}

As our proof aims to get local in-time estimates only, let us fix once and for all a positive time $T>0$.
In all that follows, $T$ can be arbitrarily large, 
but we do not want to bother with $L^{\infty}$ functions not being integrable in time.
We denote by $\Omega_T$ the space-time domain :
\begin{equation}
\Omega_T := [0,T]\times (0,1).
\end{equation}

Let us first remark that the equation \eqref{e:C-H-stand} is equivalent to the system
\begin{subequations}
\label{e:C-H-fort}
\begin{align}
\label{e:trans-fort-strech}
\partial_t y + v\cdot \partial_x y   &= -2y\partial_x v   &\text{ on } \Omega_T,       \\
\label{e:ell-v-y}
(1-\partial_x^2)v                    &= y - \kappa        &\text{ on } \Omega_T.
\end{align}
\end{subequations}
The function $y$ is called the \textit{momentum} associated with $v$.
The first equation \eqref{e:trans-fort-strech} is a \textit{transport equation} with additional \textit{stretching term},
while the second \eqref{e:ell-v-y} is elliptic.
Once written under this form, the system is analogous to the \textit{incompressible Euler equation} into vorticity form
\begin{subequations}
\begin{align}
\partial_t u + (u\cdot \nabla)\omega &= (\omega\cdot\nabla) u, \\
\div u &= 0, \\
\curl u &= \omega.
\end{align}
\end{subequations}
For the 2D Euler equation, it is common to prescribe the flux $u\cdot n$ at the boundary ($n$ is the normal vector to the boundary) as well as the entering vorticity $\omega$ see for example \cite{yudo-given}.
Using this analogy, we prescribe the flux $v$ on the boundary as well as the \textit{entering momentum}.

Let $v_l,v_r\in C^0(0,T)$.
We make the assumption that $v_r$ and $v_l$ are non-zero except on a finite set.
Let us define $\Gamma_l$ and $\Gamma_r$ similarly to \cite{Perrollaz} by
\begin{equation}\label{e:Gl-Gr-def}
\Gamma_l := \{t\in [0,T]; v_l>0\} \text{ and } \Gamma_r := \{t\in [0,T]; v_r<0\}.
\end{equation}
Those sets correspond to the sets of times $t$ where the flux is entering the domain at $x=0$ or at $x=1$.
Due to our assumption, the sets $\Gamma_l$ and $\Gamma_r$ can both be written as a finite union of open intervals.

Let $(y_l^c,y_r^c)\in L^{\infty}(\Gamma_l)\times L^{\infty}(\Gamma_r)$ and $y_0\in L^{\infty}(0,1)$.
Following \cite{Perrollaz}, we work with the initial condition
\begin{equation}\label{e:ini-cond}
y_{|t=0} = y_0
\end{equation}
as well as the boundary conditions 
\begin{subequations}
\label{e:bound-cond}
\begin{align}
\label{e:v-bound-cond}
v_{|x=0} &= v_l   \text{ on } (0,T),   \quad   \text{ and }    \quad    v_{|x=1} = v_r \text{ on } (0,T),      \\
\label{e:y-bound-cond}
y_{|x=0} &= y_l^c \text{ on }  \Gamma_l,   \quad\quad \,\,  \text{ and } \quad  y_{|x=1} = y_r^c \text{ on } \Gamma_r.
\end{align}
\end{subequations}
The letter $c$ refers to boundary condition.

The elliptic system \eqref{e:ell-v-y}-\eqref{e:v-bound-cond} can be solved to express $v$ as a function of $y$, $\kappa$~and~$(v_l,v_r)$, for $(t,x)\in \Omega_T$,
\begin{align}\notag
v(t,x) 
    &:= \cosh(x)v_l
     + \int_{0}^{x}{\cosh(x-s)Y(t,s)\mathrm{ds}} \\
\label{e:v-y-link}
    &+\tfrac{\sinh(x)}{\sinh(1)}\left(v_r-\cosh(1)v_l-\int_{0}^{1}{\cosh(1-s)Y(t,s)\mathrm{ds}}\right),
\end{align}
where 
\begin{equation}\label{e:primitives}
Y(t,x):= \int_{0}^{x}{y(t,s)\mathrm{ds}}-\kappa x
\end{equation}
Remark that this expression make sense and define a function $v\in C^0([0,T]\times [0,1])$ as soon as $y\in C^0([0,T],L^{1}(0,1))$ and $v_r,v_l\in C^0(0,T)$. \\

We give our definition of a solution to the system \eqref{e:C-H-fort}-\eqref{e:ini-cond}-\eqref{e:bound-cond}.

\begin{defn}\label{d:weaksol}
We say that a triple $(y,y_r,y_l)\in L^{\infty}(\Omega_T)\cap C^0([0,T],L^2(0,1)) \times L^{\infty}([0,T])^2$
is a weak solution of the Camassa-Holm equation with initial and boundary conditions 
$$(y_0,v_l,v_r,y_r^c,y_l^c)\in L^{\infty}(0,1)\times C^0([0,T])^2\times L^{\infty}(\Gamma_l)\times L^{\infty}(\Gamma_r),$$
when we have the compatibility conditions corresponding to \eqref{e:ini-cond} and \eqref{e:y-bound-cond}
\begin{equation}
y_{|t=0} = y_0, \quad
(y_r)_{|\Gamma_r} = y_r^c 
\quad \text{and} \quad
(y_l)_{|\Gamma_l}=y_l^c,
\end{equation}
and when for all $0\leq t_0\leq t_1 <T$ and for all test function $\phi\in H^1([t_0,t_1]\times[0,1])$,
we have:
\begin{align}
\notag
\int_{t_0}^{t_1}&{\int_{0}^{1}{\big(y\,\partial_t \phi+ yv\,\partial_x \phi - y\partial_x v\,\phi\big)}}    \\
\label{e:weakform-CH}
    &= \int_{t_0}^{t_1}{\big(y_rv_r \phi(\cdot,1)-y_lv_l\phi(\cdot,0) \big)}
    +\int_{0}^{1}{\phi(t_1,\cdot)y(t_1,\cdot)}-\int_{0}^{1}{\phi(t_0,\cdot)y(t_0,\cdot)},
\end{align}
where the function $v$ is given by the formula \eqref{e:v-y-link}.
\end{defn}

\begin{rem}
If $(y,y_r,y_l)$ is a weak solution with $y$ smooth, then $y_{\vert x=0}=y_l$ and $y_{\vert x=1}=y_r$.
\end{rem}

\begin{rem}
Any solution $y\in L^{\infty}(\Omega_T)$ in the sense of distribution of equation \eqref{e:trans-fort-strech} for a transporting vector field $v\in L^1([0,T],H^2(0,1))$ is in $C^0([0,T],L^2(0,1))$.
To look at a complete discussion on the subject, refer to \cite{Boyer}.
\end{rem}

\begin{rem}
Due to its low regularity, we cannot define the trace of $y$ in $0$ and $1$ via the standard trace theorems.
However, any distribution solution of a transport equation admits a trace in a weaker sense as long as the transporting field does not vanish on the boundary. 
To look at a complete discussion on the subject, refer to \cite{Boyer}.
\end{rem}

We can state our main theorem.
\begin{thm}\label{t:main}
Let $(y^1,y_r^1,y_l^1)$ and $(y^2,y_r^2,y_l^2)$ be two weak solutions of the Camassa-Holm equation with the same boundary conditions $(v_l,v_r,y_r^{c},y_l^{c})$ and initial conditions $y_0^1$ and $y_0^2$.
Let us assume that $v_l,v_r\in H^1(0,T)$.
Then there exists $C>0$ such that for any $0\leq T_0<T_1\leq T$ if neither $v_l$ nor $v_r$ change sign on the interval $[T_0,T_1]$, then one has the estimate
\begin{equation*}
\Vert (v^1-v^2)(T_1,\cdot)\Vert_{H^1}^2 
    \leq \left(\Vert (v^1-v^2)(T_0,\cdot)\Vert_{H^1}^2 
         + |\partial_x\tv(T_0,0)|^2
         + |\partial_x\tv(T_0,1)|^2\right) \exp(C(T_1-T_0)).
\end{equation*}
In particular, if $y_0^1=y_0^2$, then
$$(y^1,y_r^1,y_l^1)=(y^2,y_r^2,y_l^2)$$
on the interval $[0,T]$.
\end{thm}

\begin{rem}
The existence of a weak solution to the Camassa-Holm equation in the sense of Definition \ref{d:weaksol}, as well as a weak-strong uniqueness property, were tackled in \cite{Perrollaz}.
\end{rem}

\subsection{Definitions and main results for the higher order Camassa-Holm equation}

Let $n\geq 1$ be an integer.
We define the operator $\An$ by
\begin{equation}
\An := \sum_{k=0}^n{(-\partial_x^{2})^k}.
\end{equation}
For example, the operator $A_1$ is equal to 
\begin{equation}
A_1= \mathrm{Id}-\partial_x^2  ,
\end{equation}
which is the elliptic operator used to describe the standard Camassa-Holm equation, see \eqref{e:C-H-fort}.

We say that a couple $(v,y)$ is a solution of the \textit{higher order Camassa-Holm equation}  when
\begin{subequations}
\label{e:CHgen-fort}
\begin{align}
\label{e:trans-fort-strech-gen}
\partial_t y + v\partial_x y     &= -2y\partial_x v     &\text{ on } (0,1),       \\
\label{e:ell-v-y-gen}
\An v                                  &= y          &\text{ on } (0,1).
\end{align}
\end{subequations}

\begin{rem}
As mentioned in the presentation of the models,
the higher order Camassa-Holm equation were introduced by Constantin and Kolev in \cite{Constantin-Kolev}.
In their initial formulation, they are written on the torus as
\begin{equation}
\partial_t v = B_n(v,v),
\end{equation}
where $v$ is the unknown and $B_n$ is defined through
\begin{subequations}
\begin{align}
B_n(u,v) &:= - \An^{-1}(2\partial_x v \An(u)+v \An(\partial_x u)), \\
\An      &:= \sum_{k=0}^{n}{(-\partial_x^2)^k}.
\end{align}
\end{subequations}
By introducing the \textit{momentum} $y:=\An(v)$, 
we obtain the formulation \eqref{e:CHgen-fort}, 
which suits us more in the context of a boundary value problem.
\end{rem}

We prescribe the velocity $v$ on the boundary $\{0,1\}$.
We also prescribe the momentum $y$ on the part of the boundary where $v_l>0$ or $v_r<0$.
Moreover, the elliptic problem \eqref{e:ell-v-y-gen} is of order $n$.
Therefore, we prescribe more derivatives of $v$ at the boundary.
\begin{subequations}
\label{e:syst-v-y-ell-gen}
\begin{align}
\An v                             &= y,                    \\ \label{e:v-y-ell-gen-BCl}
(\S_i(v)(0))_{i\in [\![0,n-1]\!]} &= \mathbf{v_l} ,        \\ \label{e:v-y-ell-gen-BCr}
(\S_i(v)(1))_{i\in [\![0,n-1]\!]} &= \mathbf{v_r} , 
\end{align}
\end{subequations}
where the operators $\S_i$ are defined through
\begin{equation}
\forall x\in \{0,1\}, \quad  \S_i(g)(x) = \partial_x^ig(x)
\end{equation}
With that in mind let us head to the definition of weak solutions. \\

\begin{defn}[Variational solution to the Elliptic problem] \label{d:syst-v-y-ell-gen}
Let $\mathbf{v_l},\mathbf{v_r}\in\R^n$ 
and $y\in H^{-n}([0,1]):=H^n_0(0,1) '$.
Let $\chi\in C^{\infty}([0,1])$ be a smooth function equal to $0$ in a neighborhood of $1$ and equal to $0$ in a neighborhood of $1$.
We define $b=b(\mathbf{v_l},\mathbf{v_r},\chi)$ through
\begin{equation}
b(\mathbf{v_l},\mathbf{v_r},\chi)(x) = \sum_{k=0}^{n-1}{\left(\frac{x^k\chi(x)}{k!}\mathbf{v_l}_k+\frac{(1-x)^k\chi(1-x)}{k!}\mathbf{v_r}_k\right)}.
\end{equation}
We say that $v\in H^n(0,1)$ is a solution of the system \eqref{e:syst-v-y-ell-gen},
when $v-b(\mathbf{v_l},\mathbf{v_r},\chi)$ belongs to $H^n_0(0,1)$ (closure of $C^{\infty}_c(0,1)$ for the $H^n$ norm) and when for all $g\in H^n_0(0,1)$, one has 
\begin{equation}
\int_{0}^{1}{(y-\An b)g}=\int_{0}^{1}{\And (v-b)\cdot  \And g}.
\end{equation}
Where we define the operator $\And$ through
\begin{equation}
\And := (Id,\partial_x,...,\partial_x^n),
\end{equation}
and $\cdot$ is the standard scalar product in $\R^n$.
\end{defn}

\begin{rem}
This definition does not depend on the choice of $\chi$.
Moreover, thanks to Lemma \ref{l:IPP-itérée}, any smooth solution of the system \eqref{e:syst-v-y-ell-gen} is also a variational solution for this system.
\end{rem}

\begin{lem}\label{l:ell-reg-gen}
Let $y\in H^{-n}([0,1])$ be a function, and $\mathbf{v_l},\mathbf{v_r}\in\R^n$.
There exists a unique solution $v$ to the problem \eqref{e:syst-v-y-ell-gen} in the sense of Definition \ref{d:syst-v-y-ell-gen}.
This solution verifies the estimate: 
\begin{equation}
\Vert v\Vert_{H^n} 
    \lesssim \Vert y\Vert_{H^{-n}} 
     + |\mathbf{v_l}|
     + |\mathbf{v_r}|.
\end{equation}
Moreover for $y\in L^{\infty}(0,1)$, one has
\begin{equation}
\Vert v\Vert_{W^{2n,\infty}} 
    \lesssim \Vert y\Vert_{L^{\infty}} 
     + |\mathbf{v_l}|
     + |\mathbf{v_r}|.
\end{equation}
\end{lem}

\begin{proof}
The existence and uniqueness comes from Lax-Milgram's theorem. 
The $H^{-n}-H^n$ estimates is straightforward.
The $L^{\infty}-W^{2n,\infty}$ estimates comes from Lemma \ref{l:tech-ODE} below.
\end{proof}

\begin{rem}
Note that this last estimate would no longer hold in higher dimensions.
However, Schauder's estimates would give us $L^p-W^{2n,p}$ estimates for all $p$.
\end{rem}

\begin{lem}\label{l:tech-ODE}
Let $k\in \N$ be an integer, $a_0,...,a_{k-1}\in\R$ real numbers, $g\in L^{\infty}(0,1)$ a function.
Let $f$ be a solution of the ODE
\begin{equation}\label{e:ODE}
f^{(k)}+\sum_{i=0}^{k}{a_if^{(i)}} = g 
\quad\text{ on } (0,1).
\end{equation}
Then $f$ belongs to $W^{k,\infty}(0,1)$.
\end{lem}

\begin{proof}
By subtracting a combination of solutions of the homogeneous equation, which is $C^{\infty}$, we can assume that
\begin{equation}
f(0)=f^{(1)}(0)=\dots = f^{(k-1)}(0)=0.
\end{equation}
We denote by $\mathcal{P}$ the primitivation operator
\begin{equation}
\mathcal{P}f(x) := \int_{0}^{x}{f(x')\, \mathrm{d}x'}.
\end{equation}
We apply the operator $\mathcal{P}$ on equation \eqref{e:ODE} $k$ times 
\begin{equation}
f = -\sum_{i=0}^{k}{a_i\mathcal{P}^{k-i}f} + \mathcal{P}^kg.
\end{equation}
This allows ending the proof by means of a bootstrap argument.
\end{proof}
\begin{defn}[Weak solution to the higher-order CH system] \label{d:weaksol-gen}
Let $v_l,v_r\in C^0(0,T)$.
We make the assumption that $v_r$ and $v_l$ are non-zero except on a finite set.
We define the set $\Gamma_l$ and $\Gamma_r$ by \eqref{e:Gl-Gr-def}.

We say that boundary conditions
$$(\mathbf{v_l},\mathbf{v_r})\in  C^0([0,T])^n\times C^0([0,T])^n$$
are \textit{admissible}, with respect to $(v_l,v_r)$, when their first respective components $(\mathbf{v_l})_1$ and $(\mathbf{v_r})_1$ are equals respectively to $v_r$ and $v_l$
\begin{equation}
(\mathbf{v_l})_1 = v_l   \text{ and }   (\mathbf{v_r})_1=v_r.
\end{equation}

We say that a triple 
$$(y,y_r,y_l)\in L^{\infty}(\Omega_T)\cap C^0([0,T],L^2(0,1))\times L^{\infty}([0,T])^2$$
is a weak solution of the higher-order Camassa-Holm equation with initial and boundary conditions 
\begin{equation*}
(y_0,\mathbf{v_l},\mathbf{v_r},y_r^c,y_l^c)
    \in L^{\infty}(0,1)
      \times C^0([0,T])^n
      \times C^0([0,T])^n
      \times L^{\infty}(\Gamma_l)
      \times L^{\infty}(\Gamma_r),    
\end{equation*}
when we have the compatibility condition
\begin{equation}
(y_r)_{|\Gamma_r} = y_r^c \text{ and } (y_l)_{|\Gamma_l}=y_l^c,
\end{equation}
and when for all $0\leq t_0\leq t_1 <T$ and for all test function $\phi\in H^1_0([t_0,t_1]\times[0,1])$,
we have:
\begin{align}
\notag
\int_{t_0}^{t_1}&{\int_{0}^{1}{\big(y\,\partial_t \phi+ yv\,\partial_x \phi - y\partial_x v \phi\big)}}    \\
\label{e:weakform-CH_gen}
    &= \int_{t_0}^{t_1}{\big(y_rv_r \phi(\cdot,1)-y_lv_l\phi(\cdot,0) \big)}
    +\int_{0}^{1}{\phi(t_1,\cdot)y(t_1,\cdot)}-\int_{0}^{1}{\phi(t_0,\cdot)y(t_0,\cdot)},
\end{align}
where the function $v$ is given as the unique solution of the elliptic problem \eqref{e:syst-v-y-ell-gen}.
\end{defn}

With that in mind, we can formulate a local in-time existence theorem as follows.

\begin{thm}\label{t:sec}
Let $v_l,v_r\in C^0(0,T)$.
We make the assumption that $v_r$ and $v_l$ are non-zero except on a finite set.
We define the set $\Gamma_l$ and $\Gamma_r$ through \eqref{e:Gl-Gr-def}.
Let 
$(y_0,\mathbf{v_l},\mathbf{v_r},y_r^c,y_l^c)
    \in L^{\infty}(0,1)
      \times C^0([0,T])^n
      \times C^0([0,T])^n
      \times L^{\infty}(\Gamma_l)
      \times L^{\infty}(\Gamma_r)$ 
be admissible initial and boundary conditions associated with $v_r$ and $v_l$ (meaning that $(\mathbf{v_l}, \mathbf{v_r})$ is admissible with respect to $(v_r,v_l)$).

Then there exists $\tilde{T}>0$ and a weak solution 
$(y,y_r,y_l) 
    \in L^{\infty}([0,\tilde{T}]\times[0,1])\cap C^0([0,\tilde{T}],L^{2}((0,1)))
      \times L^{\infty}([0,\tilde{T}])^2$ 
to the higher-order Camassa-Holm equation 
with $(y_0,\mathbf{v_l},\mathbf{v_r},y_r^c,y_l^c)$ as initial and boundary data.
\end{thm}

The proof of Theorem \ref{t:sec} follows the lines of the one of Theorem 1 in \cite{Perrollaz}.
We give the sketch of the proof in Appendix \ref{s:app-exist} and refer to the article \cite{Perrollaz} for the detail.
As for the 3D Euler equation, the proof of existence only constructs solutions on small time intervals (see \cite{Bertozzi-Majda}).

Assuming however the existence of a solution on a given interval $[0,T]$, this solution is unique.

\begin{thm}\label{t:third}
Let $(y^1,y_r^1,y_l^1)$ and $(y^2,y_r^2,y_l^2)$ be two solutions in the sense of Definition \ref{d:weaksol-gen} with the same initial and boundary conditions
$(y_0,\mathbf{v_l},\mathbf{v_r},y_r^c,y_l^c)$.
Then 
$$(y^1,y_r^1,y_l^1)=(y^2,y_r^2,y_l^2)$$
on the interval $[0,T]$.
\end{thm}

\begin{rem}
If we compare Theorem \ref{t:main} and Theorem \ref{t:third} in the case $n=1$, the later does not have the hypothesis that $v_l$ and $v_r$ have to be $H^1$ in time, but it does not give quantitative estimates on the derivative of the flux terms at the boundary. 
\end{rem}

\subsection{Sketch of the proofs}

To prove theorem \ref{t:main}, we take two solution  $(y^1,y_r^1,y_l^1)$ and $(y^2,y_r^2,y_l^2)$ with possibly different initial and boundary values,
and we analyze the dynamics time evolution on the $H^1_x$ norm of $v^1-v^2$,
where $v^1$ and $v^2$ refer to the solutions of the elliptic problems \eqref{e:syst-v-y-ell-gen} with respectively $y^1$ and $y^2$ instead of $y$.

The sketch of the proof is the following.
In Paragraph \ref{s:energ-est}, we provide an analogous energy estimate for the relative energy between two solutions.
In Paragraph \ref{s:aux-est}, we seek to control the \textit{entering energy fluxes} which where the bad boundary terms (meaning that they cannot be discarded due to their sign and are not trivially bounded by the relative energy)  in the relative energy inequality of Paragraph \ref{s:energ-est}.
To that extent, we introduce an auxiliary test function constructed as a well-chosen elliptic multiplier of the equation.
In Paragraph \ref{s:gron}, we conclude the proof with the help of a Gronwall argument.

The proof of Theorem \ref{t:third} is similar in its structure :
in Paragraph \ref{s:energ-est-gen}, we derive a relative energy inequality, 
then in Paragraph \ref{s:aux-est-gen}, we control the entering fluxes of the relative energy inequality.

\begin{rem} \label{r:diff-proof}
The proof of Theorem \ref{t:third} is similar to the proof of Theorem \ref{t:main}.
However, it has its own difficulties.
Heuristically, in the proof of Theorem \ref{t:main}, there is an Energy Inequality, looking like
\begin{equation*}
\tfrac{d}{dt}(\mathrm{Energy})
    + \mathrm{Energy Fluxes} 
    \leq Cste \, \mathrm{Energy},
\end{equation*}
and an Auxiliary Inequality
\begin{equation*}
\tfrac{d}{dt}(\mathrm{Energy Fluxes})
    \leq Cste \, (\mathrm{Energy} + \mathrm{Energy Fluxes}).
\end{equation*}
Combining the two of them we get 
\begin{equation*}
\tfrac{d}{dt}(\mathrm{Energy}+\mathrm{Energy Fluxes})
   \leq Cste \, (\mathrm{Energy}+\mathrm{Energy Fluxes}),
\end{equation*}
which allows us to conclude through the help of Gronwall's Lemma.
In the proof of Theorem \ref{t:third}, the Energy Inequality still looks like 
\begin{equation*}
\tfrac{d}{dt}(\mathrm{Energy})
    + \mathrm{ExitingFluxes} 
    - \mathrm{EnteringFluxes}
    \leq Cste \, \mathrm{Energy},
\end{equation*}
but the Auxiliary Inequality is more of the form
\begin{equation*}
\tfrac{d}{dt}(\mathrm{Auxiliary})
    + \mathrm{EnteringFluxes}
    \leq Cste \, (\mathrm{Energy} + \mathrm{Auxiliary}).
\end{equation*}
Combining the two of them we get 
\begin{equation*}
\tfrac{d}{dt}(\mathrm{Energy}+\mathrm{Auxiliary})
    + |\mathrm{Energy Fluxes}|
   \leq Cste \, (\mathrm{Energy}+\mathrm{Auxiliary}),
\end{equation*}
which still allows us to conclude through the help of Gronwall's Lemma.

The proof is quite different because of the construction of the auxiliary test function.
We solve the same dual elliptic problem to construct it, but in the case of Camassa-Holm,
the boundary data for this elliptic problem are bounded by the energy fluxes.
This is not true in the higher-order case.
That is the main reason why the proof for Camassa-Holm is easier and stronger,
i.e. it gives bounds on the derivative of the entering flux. 
\end{rem}

\begin{rem}
In the case of periodic boundary condition, the Camassa-Holm equation was proved to be locally well-posed in $H^s(0,1)$ for $1\leq s <2$, for any initial data $v_0\in W^{1,\infty}$ (see  \cite{DKT-CH}).
To define weak solutions for which the momentum $y$ is not $L^{\infty}$ or even $L^1$, one writes the equation as
\begin{subequations}
\label{e:CH-pressure}
\begin{align}
\partial_t v + v\partial_xv &= \partial_x P, \\
P &= (1-\partial_x^2)^{-1}\left(v^2+\tfrac{(\partial_xv)^2}{2}\right).
\end{align}
\end{subequations}
This is easier to define in the case of periodic boundary conditions as one does not need additional boundary condition for $P$.
To go into lower regularity than what we do in this article (meaning solution with bounded momentum), one could explore this formulation of the equation and in particular, ask ourselves which boundary conditions are needed for it to make sense.
\end{rem}

\begin{rem}
One can also remark that a formulation similar to \eqref{e:CH-pressure} exists for the higher order Camassa-Holm system \eqref{e:CHgen-fort}.
One can recast this system as :
\begin{subequations}
\begin{align}
\partial_t v + v\partial_x v &= \partial_x P, \\
A_n(P) &= \mathcal{F}_n[v]
\end{align}
\end{subequations}
where $\mathcal{F}_n[v]$ is a differential polynomial in $v$ depending on $n$ that we will not describe here.

It is using this formulation as well as elliptic regularization of the equation also that Coclite, Holden, and Karlsen tackled the existence of a solution for the higher order Camassa-Holm system on the circle (see \cite{Coc-Hol-karl2}).
They also obtained a weak-strong uniqueness result.
\end{rem}

\begin{rem}
The problem of stability estimates on the whole interval when the flux $v_r$ or $v_l$ changes sign is still open. 
\end{rem}

\section{Proof of Theorem \ref{t:main}}
\subsection{Energy estimate for the difference of two solutions}\label{s:energ-est}

Let us take two weak solutions $(y^1,y_r^1,y_l^1)$ and $(y^2,y_r^2,y_l^2)$ of the Camassa-Holm equation with initial and boundary conditions $(y_0^1,v_l,v_r,y_r^{1,c},y_l^{1,c})$ and $(y_0^2,v_l,v_r,y_r^{2,c},y_l^{2,c})$.
We define the following functions 
\begin{align}
\ty   &:= y^1-y^2,                  &\tv  := v^1-v^2,             \\
\hy   &:= \frac{y^1+y^2}{2},        &\hv  := \frac{v^1+v^2}{2},   \\
\tyl  &:= y_l^1-y_l^2,              &\tyr := y_r^1-y_r^2,            
\end{align}
where the functions $v^1$ and $v^2$ are given through \eqref{e:v-y-link}. \\

We take the difference of Equation \eqref{e:weakform-CH} for the solutions $y^1$ and $y^2$.
The function $\ty$ verifies the following equality for all $0\leq t_0\leq t_1 <T$ and for all test function $\phi\in H^1([t_0,t_1]\times[0,1])$:
\begin{align}
\notag
\int_{t_0}^{t_1}&{\int_{0}^{1}{\big(\ty\,\partial_t \phi+ (\ty\hv+\hy\tv)\,\partial_x \phi - (\ty\partial_x \hv+\hy\partial_x\tv)\phi\big)}}        \\
\label{e:weakform-CH-diff}
    &= \int_{t_0}^{t_1}{\tyr v_r \phi(\cdot,1)} - \int_{t_0}^{t_1}{\tyl v_l \phi(\cdot,0)}
    +\int_{0}^{1}{\phi(t_1,\cdot)\ty(t_1,\cdot)}-\int_{0}^{1}{\phi(t_0,\cdot)\ty(t_0,\cdot)}.
\end{align}

Furthermore the functions $\tv$ and $\hv$ are solutions of the following elliptic problems:
\begin{subequations} \label{e:syst-ell-tv-ty}
\begin{align} \label{e:ell-tv-ty}
&(1-\partial_x^2)\tv = \ty           \text{ on } (0,1), 
&(1-\partial_x^2)\hv = \hy - \kappa  \text{ on } (0,1), \\
&\tv_{|x=0} = 0, 
&\hv_{|x=0} = v_l, \\
&\tv_{|x=1} = 0,
&\hv_{|x=1} = v_r.
\end{align}
\end{subequations}
With that in mind, we prove the following lemma.

\begin{lem}\label{l:tv-time-reg}
The functions $\hv$ and $\tv$ lie in $L^{\infty}([0,T],W^{2,\infty}([0,1]))$.
Moreover, the function $\tv$ lies in $W^{1,\infty}([0,T],H^1([0,1]))$.
\end{lem}

\begin{proof}
To obtain the space regularity of $\hv$ and $\tv$, remark that the primitives $Y^1$ and $Y^2$, defined as in \ref{e:primitives} with $y^1$ and $y^2$ instead of $y$, are $L^{\infty}([0,T],W^{1,\infty}(0,1))$.
Then use the formula \eqref{e:v-y-link}. 

Let us now prove the time regularity of $\tv$.
Let us fix two times $t_0<t_1$, and denote 
$$a_{t_0}^{t_1}(x):=\tv(t_1,x)-\tv(t_0,x).$$
Recalling that $\tv$ verifies \eqref{e:ell-tv-ty} with homogeneous boundary conditions, we obtain that
\begin{equation*}
\int_{0}^{1}{|\partial_x a_{t_0}^{t_1}|^2}
    =-\int_{0}^{1}{\big(a_{t_0}^{t_1}-(\ty(t_1,\cdot)-\ty(t_0,\cdot))\big)a_{t_0}^{t_1}},   
\end{equation*}
which can be rewritten as
\begin{equation*}
\int_{0}^{1}{|\partial_x a_{t_0}^{t_1}|^2}
    +\int_{0}^{1}{|a_{t_0}^{t_1}|^2}
    =\int_{0}^{1}{(\ty(t_1,\cdot)-\ty(t_0,\cdot))a_{t_0}^{t_1}} \, .   
\end{equation*}
Hence, we get the inequality
\begin{equation}\label{e:a-H1norm}
\Vert a_{t_0}^{t_1}\Vert_{H^1(0,1)}^2 
    \leq  \left\vert\int_{0}^{1}{(\ty(t_1,\cdot)-\ty(t_0,\cdot))a_{t_0}^{t_1}}\right\vert.
\end{equation}
Using \eqref{e:weakform-CH-diff} with $a_{t_0}^{t_1}$ instead of $\phi$ (considered as a function constant in time), we obtain that:
\begin{equation}\label{e:a-test-func}
\int_{0}^{1}{(\ty(t_1,\cdot)-\ty(t_0,\cdot))a_{t_0}^{t_1}}
    = \int_{t_0}^{t_1}{\int_{0}^{1}{\big((\ty\hv+\hy\tv)\partial_x a_{t_0}^{t_1} - (\ty\partial_x \hv+\hy\partial_x\tv)a_{t_0}^{t_1}\big)}}.
\end{equation}
Combining \eqref{e:a-H1norm} and \eqref{e:a-test-func}, we get that:
\begin{equation*}
\Vert a_{t_0}^{t_1}\Vert_{H^1(0,1)}
    \leq |t_1-t_0|\big(\Vert\ty\Vert_{L^{\infty}([0,T]\times[0,1])}\Vert\hv\Vert_{L^{\infty}([0,T],H^1(0,1))}
    +\Vert\hy\Vert_{L^{\infty}([0,T]\times[0,1])}\Vert\tv\Vert_{L^{\infty}([0,T],H^1(0,1))}\big)
\end{equation*}
Recalling that $a_{t_0}^{t_1}(x)=\tv(t_1,x)-\tv(t_0,x)$, we conclude that $\tv\in W^{1,\infty}([0,T],H^1([0,1]))$.
\end{proof}

Now that we have Lemma \ref{l:tv-time-reg}, we prove the following relative energy equality :

\begin{prop}\label{p:energ-eq}
For all $0\leq t_0<t_1\leq T$, we have the following equality:
\begin{align}
\notag
\Vert\tv(t_1,\cdot)\Vert_{H^1(0,1)}^2
    &-\Vert\tv(t_0,\cdot)\Vert_{H^1(0,1)}^2
    +\int_{t_0}^{t_1}{|\partial_x \tv(t,1)|^2v_r} 
    -\int_{t_0}^{t_1}{|\partial_x \tv(t,0)|^2v_l} \\
\label{e:energ-eq}
    &+\int_{t_0}^{t_1}{\int_{0}^{1}{\left(3|\tv|^2+|\partial_x \tv|^2\right)\partial_x\hv}}
    +\int_{t_0}^{t_1}{\int_{0}^{1}{\partial_x\left(|\tv|^2\right)\left(\hv-\hy+\kappa\right)}}
    =0.
\end{align}
\end{prop}

\begin{proof}
Thanks to Lemma \ref{l:tv-time-reg}, we can take $\tv$ as a test function in \eqref{e:weakform-CH-diff}, which we do.
For all $t_0<t_1$, we have that
\begin{align}
\notag
\int_{t_0}^{t_1}&{\int_{0}^{1}{\big(\ty\,\partial_t \tv+ \ty\hv\partial_x \tv - \ty\tv\partial_x \hv\big)}}        \\
\label{e:weakform-tv-test1}
    &= \int_{t_0}^{t_1}{\tyr v_r \tv(\cdot,1)} -\int_{t_0}^{t_1}{\tyl v_l \tv(\cdot,0)}
    +\int_{0}^{1}{\tv(t_1,\cdot)\ty(t_1,\cdot)}-\int_{0}^{1}{\tv(t_0,\cdot)\ty(t_0,\cdot)}.
\end{align}
We cancel the boundary terms, because $\tv_{|x = 0} = \tv_{|x = 1} = 0$, to get
\begin{equation}\label{e:weakform-tv-test2}
\int_{t_0}^{t_1}{\int_{0}^{1}{\big(\ty\,\partial_t\tv+ \ty\hv\partial_x\tv - \ty\partial_x\hv\tv\big)}}
    = \int_{0}^{1}{\tv(t_1,\cdot)\ty(t_1,\cdot)}-\int_{0}^{1}{\tv(t_0,\cdot)\ty(t_0,\cdot)}.
\end{equation}
Now, we reformulate each term of \eqref{e:weakform-tv-test2} by using some integration by parts as well as \eqref{e:syst-ell-tv-ty}.
\begin{itemize}
\item
First let us look at $\int_{0}^{1}{\tv(t,\cdot)\ty(t,\cdot)}$ (which will be used for $t=t_0$ and $t=t_1$)
\begin{align}
\notag
\int_{0}^{1}{\tv(t,\cdot)\ty(t,\cdot)} 
    &= \int_{0}^{1}{\tv(1-\partial_x^2)\tv} \\
\notag
    &= \int_{0}^{1}{|\tv|^2}                
      +\int_{0}^{1}{|\partial_x\tv|^2}      \\
\label{e:energ-eq-prf1}
    &=\Vert\tv(t,\cdot)\Vert_{H^1(0,1)}^2.
\end{align}

\item 
The term $\int_{t_0}^{t_1}{\int_{0}^{1}{\ty\,\partial_t \tv}}$ is dealt with similarly 
\begin{equation}\label{e:energ-eq-prf2}
\int_{t_0}^{t_1}{\int_{0}^{1}{\ty\,\partial_t \tv}} 
    =\frac{1}{2}\left(\Vert\tv(t_1,\cdot)\Vert_{H^1(0,1)}^2-\Vert\tv(t_0,\cdot)\Vert_{H^1(0,1)}^2\right).
\end{equation}

\item 
Now, we deal with the two bilinear terms.
Let us recast the first one:
\begin{align}
\notag
\int_{t_0}^{t_1}{\int_{0}^{1}{\ty\hv\partial_x \tv}}
    &= \int_{t_0}^{t_1}{\int_{0}^{1}{(1-\partial_x^2)\tv \hv\partial_x \tv}} \\
\notag 
    &= \frac{1}{2}\int_{t_0}^{t_1}{\int_{0}^{1}{\partial_x\left(|\tv|^2-|\partial_x \tv|^2\right)\hv}} \\
\label{e:energ-eq-prf3}
    &= -\frac{1}{2}\int_{t_0}^{t_1}{\int_{0}^{1}{\left(|\tv|^2-|\partial_x \tv|^2\right)\partial_x\hv}} 
     -\frac{1}{2}\int_{t_0}^{t_1}{\left(|\partial_x \tv(t,1)|^2v_r-|\partial_x \tv(t,0)|^2v_l\right)}.
\end{align}

\item
Let us reformulate the second one:
\begin{align}
\notag
\int_{t_0}^{t_1}{\int_{0}^{1}{\ty\tv\partial_x \hv}}
    &= \int_{t_0}^{t_1}{\int_{0}^{1}{(1-\partial_x^2)\tv \tv\partial_x \hv}} \\
\label{e:energ-eq-prf4}
    &= \int_{t_0}^{t_1}{\int_{0}^{1}{(|\tv|^2\partial_x \hv + |\partial_x \tv|^2\partial_x \hv +\tv\partial_x\tv \partial_x^2\hv)}},
\end{align}
and we use \eqref{e:ell-tv-ty} to get rid of the second order derivative:
\begin{equation}\label{e:energ-eq-prf5}
\int_{t_0}^{t_1}{\int_{0}^{1}{\tv\partial_x\tv \partial_x^2\hv}} 
    =\tfrac{1}{2}\int_{t_0}^{t_1}{\int_{0}^{1}{\partial_x\left(|\tv|^2\right)\left(\hv-\hy+\kappa\right)}}.
\end{equation}
\end{itemize}

By combining \eqref{e:weakform-tv-test2} with \eqref{e:energ-eq-prf1}, \eqref{e:energ-eq-prf2}, \eqref{e:energ-eq-prf3}, \eqref{e:energ-eq-prf4} and \eqref{e:energ-eq-prf5}, we obtain the wanted result.
\end{proof}

We deduce the following corollary.

\begin{cor}\label{c:energ-ineq}
There exists a constant $C>0$ such that for almost every $0<t<T$, we have the following inequality:
\begin{equation}\label{i:energ-ineq}
\frac{d}{dt}\left(\Vert\tv(t,.)\Vert_{H^1(0,1)}^2\right)  
    + |\partial_x \tv(t,1)|^2v_r(t) - |\partial_x \tv(t,0)|^2v_l(t)
    \leq C \Vert\tv(t,.)\Vert_{H^1(0,1)}^2.
\end{equation}
\end{cor}

\begin{proof}
One starts from equality \ref{e:energ-eq}, with $t_0=t-\varepsilon$ and $t_1=t+\varepsilon$ for $t\in [0,T]$ and $\varepsilon>0$.

Since $\tv$ lies in $W^{1,\infty}_tH^1_x$, the fraction
\begin{equation*}
\frac{\Vert\tv(t+\varepsilon,.)\Vert_{H^1(0,1)}^2-\Vert\tv(t-\varepsilon,.)\Vert_{H^1(0,1)}^2}{\varepsilon}
\end{equation*}
converges for almost every~$t$ towards 
\begin{equation*}
\frac{d}{dt}\left(\Vert\tv(t,.)\Vert_{H^1(0,1)}^2\right).
\end{equation*}

The quantities $\partial_x\tv(\cdot,0)^2v_l$ and $\partial_x\tv(\cdot,1)^2v_r$ are both $L^{\infty}$.
Therefore, by Rademacher's Theorem, for almost every $t\in [0,T]$, the integral
\begin{equation*}
\frac{1}{\varepsilon}\int_{t-\varepsilon}^{t+\varepsilon}{\partial_x\tv(\cdot,0)^2v_l}
\end{equation*}
converges towards 
\begin{equation*}
\partial_x\tv(t,0)^2v_l(t).
\end{equation*}
Similarly
$\tfrac{1}{\varepsilon}\int_{t-\varepsilon}^{t+\varepsilon}{\partial_x\tv(\cdot,1)^2v_r}$
converges towards $\partial_x\tv(t,1)^2v_r(t)$.

Moreover, using the Cauchy-Schwarz inequality, one gets that for all $t\in [0,T]$ and for all $\varepsilon>0$
\begin{align*}
\int_{t-\varepsilon}^{t+\varepsilon}{\int_{0}^{1}{\left(3|\tv|^2+|\partial_x \tv|^2\right)\partial_x\hv}}
    &\leq 4\varepsilon
       \Vert\hv\Vert_{L^{\infty}_tW^{1,\infty}_x} 
       \Vert\tv\Vert_{L^{\infty}([t-\varepsilon,t+\varepsilon],H^1(0,1))}^2, \\
\int_{t-\varepsilon}^{t+\varepsilon}{\int_{0}^{1}{\partial_x\left(|\tv|^2\right)\left(\hv-\hy+\kappa\right)}}
    &\leq \varepsilon 
       \left(\Vert\hv\Vert_{L^{\infty}_tL^{\infty}_x} 
           + \Vert\hy\Vert_{L^{\infty}_tL^{\infty}_x}
           + \kappa\right)
       \Vert\tv\Vert_{L^{\infty}([t-\varepsilon,t+\varepsilon],H^1(0,1))}^2.
\end{align*}
\end{proof}
%


\subsection{Auxiliary inequality}\label{s:aux-est}

We define two functions $u_l$ and $u_r$ by setting for all $x\in [0,1]$
\begin{equation}
u_l(x):=-\sinh(x) + \cosh(x) \tanh(1) 
\quad \text{ and } \quad
u_r(x) := -\frac{\sinh(x)}{\cosh(1)}.
\end{equation}
They are the solutions to the non-homogeneous Zaremba-type problems:
\begin{subequations}
\begin{align}
(1-\partial_x^2)u_l  &=          (1-\partial_x^2)u_r                     = 0 ,                       \\
-\partial_x u_l (0)  &=  \,\,\,  \partial_x u_r (1)    \,\,\,\,          = 1 ,                       \\
u_l (1)              &=  \quad   u_r (0)               \quad \,          = 0.
\end{align}
\end{subequations}

We want to bound $\partial_x\tv$ at the boundary with the help of a Gronwall argument.
Let us begin by showing that $\partial_x\tv(\cdot,0)$ and $\partial_x\tv(\cdot,1)$ are a lipschitz functions.

%
\begin{lem}\label{l:tva-time-reg}
The functions $\partial_x \tv(\cdot,0)$ and $\partial_x \tv(\cdot,1)$ are lipschitz functions with respect to time.
\end{lem}

\begin{proof}
By differentiating \eqref{e:v-y-link} in $x$, we obtain:
\begin{equation}
\partial_x \tv(t,x)= 
    \int_{0}^{x}{\sinh(x-s)\tilde{Y}(t,s)\mathrm{ds}}
    + \cosh(0)\tilde{Y}(t,x)
    -\tfrac{\cosh(x)}{\sinh(1)}\int_{0}^{1}{\cosh(1-s)\tilde{Y}(t,s)\mathrm{ds}}.
\end{equation}
To prove the regularity in time of $\partial_x \tv(\cdot,0)$, we prove the time regularity of the function $\tilde{Y}$.

Let $\phi\in L^2(0,1)$ be a function.
We denote by $\Phi$ the primitive of $\phi$ verifying $\Phi(1)=0$:
\begin{equation*}
\Phi(x):= -\int_{x}^{1}{\phi(s)\mathrm{ds}}.
\end{equation*}
Using $\Phi$, (considered as a constant function in time) as a test function in \eqref{e:weakform-CH-diff}, 
we obtain that
\begin{equation}\label{e:Phi-test-func}
\int_{t_0}^{t_1}{\tyl v_l \Phi(0)}
    + \int_{0}^{1}{\Phi\ty(t_1,\cdot)}
    -\int_{0}^{1}{\Phi\ty(t_0,\cdot)}
    = \int_{t_0}^{t_1}{\int_{0}^{1}{\big( (\ty\hv+\hy\tv)\phi - (\ty\partial_x \hv+\hy\partial_x\tv)\,\Phi\big)}}.
\end{equation}
Moreover, by integration by parts, we have 
\begin{equation}\label{e:Phiy-phiY}
\int_{0}^{1}{\Phi\ty(t,\cdot)} = - \int_{0}^{1}{\phi\tilde{Y}(t,\cdot)}.
\end{equation}
Hence, by combining \eqref{e:Phi-test-func} and \eqref{e:Phiy-phiY}, we obtain that
\begin{align*}
\Vert\tilde{Y}(t_1,\cdot)-\tilde{Y}(t_0,\cdot)\Vert_{L^{2}(0,1)} 
    &\leq |t_1-t_0|\big( 
        \Vert \tyl\Vert_{L^{\infty}([0,T])} \\
        &+\Vert\ty\hv+\hy\tv\Vert_{L^{\infty}([0,T],L^2(0,1))}
        + \Vert\ty\partial_x\hv+\hy\partial_x\tv\Vert_{L^{\infty}([0,T],L^2(0,1))}
        \big).
\end{align*}
\end{proof}

We prove the following auxiliary inequalities.

\begin{prop}\label{p:aux-ineq}
There exists a constant $C>0$ such that, we have the inequalities
\begin{align}
\label{i:aux-ineq-l}
\forall^{a.e.} t\in\Gamma_l, \quad 
\frac{d}{dt}\big(|\partial_x \tv(t,0)|^2\big) 
    \leq C &\left(\Vert\tv(t,\cdot)\Vert_{H^1(0,1)}^2+|\partial_x \tv(t,0)|^2\right) \\ \notag
      &+\frac{1}{2}|\partial_x \tv(t,1)|^2|v_r(t)|
      + |\tyl|^2v_l^+,  \\
\label{i:aux-ineq-r}
\forall^{a.e.} t\in\Gamma_r, \quad 
\frac{d}{dt}\big(|\partial_x \tv(t,1)|^2\big) 
    \leq C &\left(\Vert\tv(t,\cdot)\Vert_{H^1(0,1)}^2+|\partial_x \tv(t,1)|^2\right) \\ \notag
      &+\frac{1}{2}|\partial_x \tv(t,0)|^2|v_l(t)|
      + |\tyr|^2v_r^+,
\end{align}
where we recall that $\Gamma_l/ \Gamma_r$ are the set of times of entering flux at the left/right defined in \eqref{e:Gl-Gr-def} and $v_l^+$ (resp. $v_r^+$) is the positive part of $v_l$ (resp. $v_r$).
\end{prop}

\begin{proof}
The two inequalities \eqref{i:aux-ineq-l} and \eqref{i:aux-ineq-r} have the same proof, we prove inequality \eqref{i:aux-ineq-l} here.

We define the auxiliary test function $\tval$ through
\begin{equation}
\forall t\in [0,T],\forall x\in [0,1], \quad \tval (t,x) := \partial_x\tv (t,0)\, v_l^+(t)\, u_l(x),
\end{equation}

Let $0<t_0<t_1<T$ be two positive times (at the end of the proof, we will take $t_0=t-\varepsilon$ and $t_1=t+\varepsilon$ and make $\varepsilon$ goes to $0$).
Using Lemma \ref{l:tva-time-reg}, we know that we can take $\tval$ as test function in \eqref{e:weakform-CH-diff}, which leads to
\begin{align}
\notag
\int_{t_0}^{t_1}&{\int_{0}^{1}{\big(\ty\,\partial_t \tval+ (\ty\hv+\hy\tv)\,\partial_x \tval - (\ty\partial_x \hv+\hy\partial_x\tv).\tval\big)}}        \\
\label{e:weakform-tva-test}
    &= -\int_{t_0}^{t_1}{\tyl v_l \tval(\cdot,0)}
    +\int_{0}^{1}{\tval(t_1,\cdot)\ty(t_1,\cdot)}
    -\int_{0}^{1}{\tval(t_0,\cdot)\ty(t_0,\cdot)},
\end{align}
The boundary term $\int_{t_0}^{t_1}{\tyr v_r\tval(\cdot,1)}$ is equal to $0$ due to the assumption $u_l(1)=0$, and therefore $\tval(\cdot,1)=0$.

Let us remark that the boundary term $\int_{t_0}^{t_1}{\tyl v_l \tval(\cdot,0)}$ is also equal to $0$ in the case which interests us. 
Indeed if $y_l^{1,c}=y_l^{2,c}$ (meaning that the two solutions have the same boundary condition) then $\tyl=0$.

We simplify each term similarly to the proof of Proposition \ref{p:energ-eq}. 

For $a:[0,1]\rightarrow \R$ continuous, we use the notation $[a]_0^1$ for 
\begin{equation*}
[a]_0^1:=a(1)-a(0)
\end{equation*}

\begin{itemize}
\item 
First, let us simplify $\int_{t_0}^{t_1}{\int_{0}^{1}{\ty\,\partial_t \tval}}$.
To do so, we replace $\ty$ by $(1-\partial_x^2)\tv$ using \eqref{e:ell-tv-ty}.
Then, we integrate by parts
\begin{align}
\notag
\int_{t_0}^{t_1}{\int_{0}^{1}{\ty\,\partial_t \tval}}
    = &\int_{t_0}^{t_1}{\int_{0}^{1}{(1-\partial_x^2)\tv\,\partial_t \tval}}   \\  \notag
    = &\int_{t_0}^{t_1}{\int_{0}^{1}{\tv\,\partial_t \tval}}                   
       -\int_{t_0}^{t_1}{\int_{0}^{1}{\tv\,\partial_t \partial_x^2\tval}}      \\  \notag
       &+\int_{t_0}^{t_1}{\left[\tv\,\partial_t \partial_x\tval\right]_0^1}
       -\int_{t_0}^{t_1}{\left[\partial_x \tv\,\partial_t \tval\right]_0^1}.
\end{align}
By definition of $\tval$, we have $(1-\partial_x^2)\tval = 0$, which allows us to cancel the first two terms.
Moreover $\tv_{|x=0}=\tv_{|x=1}=0$, which allows to forget the third term.
For the last term, we use the facts that 
$\tval_{|x=1}=0$,
$u_l(0)=\tanh(1)$ and
$\partial_x \tval_{|x=0}=-\partial_x \tv_{|x=0} v_l^+$.
This gives
\begin{align}
\notag
\int_{t_0}^{t_1}{\int_{0}^{1}{\ty\,\partial_t \tval}}
    =  &-\int_{t_0}^{t_1}{\left[\partial_x \tv\,\partial_t \tval\right]_0^1}    \\  \notag
    = & \tanh(1)\int_{t_0}^{t_1}{\partial_x \tv(t,0)\frac{d}{dt}(\partial_x \tv(t,0)v_l^+(t))}    \\ 
\label{e:aux-in-pcomb0}
    = & \tanh(1)\left(|\partial_x \tv(t_1,0)|^2v_l^+(t_1)-|\partial_x \tv(t_0,0)|^2v_l^+(t_0)\right) \\
\notag
     &- \frac{\tanh(1)}{2} \int_{t_0}^{t_1}{v_l^+(t)\frac{d}{dt}\left(|\partial_x \tv(t,0)|^2\right)}.
\end{align}

\item
For all $t\in[0,T]$, and in particular for $t=t_0$ and $t=t_1$, 
we simplify $\int_{0}^{1}{\ty(t,\cdot).\tval(t,\cdot)}$ similarly:
\begin{align}
\notag
\int_{0}^{1}{\ty(t,\cdot) \tval(t,\cdot)} 
    = &\int_{0}^{1}{(1-\partial_x^2)\tv(t,\cdot) \tval(t,\cdot)} \\ \notag
    = & \int_{0}^{1}{\tv\, \tval}                   
       -\int_{0}^{1}{\tv\, \partial_x^2\tval}      
       + \left[\tv\, \partial_x\tval\right]_0^1
       -\left[\partial_x \tv\, \tval\right]_0^1         \\   \notag
    = & -\left[\partial_x \tv\, \tval\right]_0^1        \\   
\label{e:aux-in-pcomb1}
    = & \tanh(1) |\partial_x \tv(t,0)|^2v_l^+(t).
\end{align}

\item
We bound $\int_{t_0}^{t_1}{\int_{0}^{1}{\hy\tv\,\partial_x \tval}}$ and $\int_{t_0}^{t_1}{\int_{0}^{1}{\hy\partial_x\tv. \tval}}$ using the Cauchy-Schwarz inequality:
\begin{align}
\left|\int_{t_0}^{t_1}{\int_{0}^{1}{\hy\tv\,\partial_x \tval}}\right| 
    &\leq \Vert\hy\Vert_{L^{\infty}(\Omega_T)} \int_{t_0}^{t_1}{\Vert\tv\Vert_{L^2(0,1)}\,\Vert\tval\Vert_{H^1(0,1)}}, \\
\left|\int_{t_0}^{t_1}{\int_{0}^{1}{\hy\partial_x \tv.\tval}}\right| 
    &\leq \Vert\hy\Vert_{L^{\infty}(\Omega_T)} \int_{t_0}^{t_1}{\Vert\tv\Vert_{H^1(0,1)}\,\Vert\tval\Vert_{L^2(0,1)}}.
\end{align}
We simplify this expression using the fact that for $a,b,c\geq 0$, one has $a^2c+b^2c\geq 2abc$
\begin{align*}
\Vert\tv(t,\cdot)\Vert_{L^2(0,1)}\,\Vert\tval(t,\cdot)\Vert_{H^1(0,1)}
    &\leq \Vert\tv(t,\cdot)\Vert_{L^2(0,1)}^2 v_l^+(t) 
      + \Vert u_l\Vert_{H^1(0,1)}^2 |\partial_x \tv(t,0)|^2v_l^+(t), \\
\Vert\tv(t,\cdot)\Vert_{H^1(0,1)}\,\Vert\tval(t,\cdot)\Vert_{L^2(0,1)}
    &\leq \Vert\tv(t,\cdot)\Vert_{H^1(0,1)}^2 v_l^+(t) 
      + \Vert u_l\Vert_{L^2(0,1)}^2 |\partial_x \tv(t,0)|^2v_l^+(t). 
\end{align*}
Therefore
\begin{align} \notag
\left|\int_{t_0}^{t_1}{\int_{0}^{1}{\hy\tv\,\partial_x \tval}}\right|
    &+ \left|\int_{t_0}^{t_1}{\int_{0}^{1}{\hy\partial_x \tv.\tval}}\right| \\ \label{e:aux-in-pcomb2}
    &\leq C \int_{t_0}^{t_1}{\left(\Vert\tv(t,\cdot)\Vert_{H^1(0,1)}^2
                                  +|\partial_x \tv(t,0)|^2\right)v_l^+(t)\,\mathrm{d}t}.
\end{align}

\item
We simplify the term $\int_{t_0}^{t_1}{\int_{0}^{1}{\ty\hv\,\partial_x \tval}}$
\begin{align}
\notag
\int_{t_0}^{t_1}{\int_{0}^{1}{\ty\hv\,\partial_x \tval}}
    &= \int_{t_0}^{t_1}{\int_{0}^{1}{(1-\partial_x^2)\tv\hv\,\partial_x \tval}} \\
\label{e:aux-in-proof1}
    &= \int_{t_0}^{t_1}{\int_{0}^{1}{\tv\hv\,\partial_x \tval}}
      -\int_{t_0}^{t_1}{\int_{0}^{1}{\partial_x^2\tv\hv\,\partial_x \tval}}.     
\end{align}
We bound the first term $\int_{t_0}^{t_1}{\int_{0}^{1}{\tv\hv\,\partial_x \tval}}$ of the right hand side of \eqref{e:aux-in-proof1} by
\begin{align}
\notag
\left|\int_{t_0}^{t_1}{\int_{0}^{1}{\tv\hv\,\partial_x \tval}}\right| 
    &\leq \Vert\hv\Vert_{L^{\infty}(\Omega_T)} \int_{t_0}^{t_1}{\Vert\tv\Vert_{L^2(0,1)}\,\Vert\tval\Vert_{H^1(0,1)}} \\
\label{e:aux-in-proof2}
    &\leq C \int_{t_0}^{t_1}{\left(\Vert\tv\Vert_{H^1(0,1)}^2+|\partial_x \tv(\cdot,0)|^2\right)v_l^+}.
\end{align}

For the second term $\int_{t_0}^{t_1}{\int_{0}^{1}{\partial_x^2\tv\hv\,\partial_x \tval}}$ of the right hand side of \eqref{e:aux-in-proof1}, we have
\begin{equation}\label{e:aux-in-proof3}
\int_{t_0}^{t_1}{\int_{0}^{1}{\partial_x^2\tv\hv\,\partial_x \tval}}
    = \int_{t_0}^{t_1}{[\partial_x\tv\partial_x\tval \hv]_0^1}
     -\int_{t_0}^{t_1}{\int_{0}^{1}{\partial_x\tv\partial_x(\hv\,\partial_x \tval)}}
\end{equation}
We bound the trilinear term using the fact that 
$\Vert\tval\Vert_{H^2}=|\partial_x \tv(\cdot,0)|v_l^+\Vert u_l\Vert_{H^2}$.
\begin{equation}\label{e:aux-in-proof4}
\left|\int_{t_0}^{t_1}{\int_{0}^{1}{\partial_x\tv\partial_x(\hv\,\partial_x \tval)}}\right|
    \leq  C \int_{t_0}^{t_1}{\left(\Vert\tv\Vert_{H^1(0,1)}^2+|\partial_x \tv(\cdot,0)|^2\right)v_l^+}.
\end{equation}
The boundary term 
$\int_{t_0}^{t_1}{(\partial_x\tv\partial_x\tval\hv)_{|x=0}}=-\int_{t_0}^{t_1}{|\partial_x\tv(\cdot,0)|^2 (v_l^+)^2}$
can be left as is (as it is negative), whereas the term $\int_{t_0}^{t_1}{(\partial_x\tv\partial_x\tval \hv)_{|x=1}}$ can be bounded through
\begin{align}
\notag
\left|\int_{t_0}^{t_1}{(\partial_x\tv\partial_x\tval \hv)_{|x=1}}\right|
    &\leq \int_{t_0}^{t_1}{\left(\tfrac{\tanh(1)}{4}|\partial_x\tv(\cdot,1)|^2|v_r|v_l^+ + \tfrac{4}{\tanh(1)}|\partial_x\tv(\cdot,0)|^2|\partial_x u_l(1)|^2|v_r|v_l^+\right)}  \\
\label{e:aux-in-proof5}
    &\leq \int_{t_0}^{t_1}{\left(\tfrac{\tanh(1)}{4}|\partial_x\tv(\cdot,1)|^2|v_r|v_l^+ + C|\partial_x\tv(\cdot,0)|^2v_l^+\right)}
\end{align}
Combining \eqref{e:aux-in-proof1}-\eqref{e:aux-in-proof5}, we get
\begin{equation}\label{e:aux-in-pcomb3}
\left|\int_{t_0}^{t_1}{\int_{0}^{1}{\ty\hv\,\partial_x \tval}}\right|
    \leq C \int_{t_0}^{t_1}{\left(\Vert\tv\Vert_{H^1(0,1)}^2+|\partial_x \tv(\cdot,0)|^2\right)v_l^+}
     + \tfrac{\tanh(1)}{4}\int_{t_0}^{t_1}{|\partial_x\tv(\cdot,1)|^2|v_r|v_l^+}.
\end{equation}

\item
We simplify the term $\int_{t_0}^{t_1}{\int_{0}^{1}{\ty\partial_x\hv. \tval}}$ of \eqref{e:weakform-tva-test}:
\begin{align}
\notag
\int_{t_0}^{t_1}{\int_{0}^{1}{\ty \tval\,\partial_x\hv}}
    &= \int_{t_0}^{t_1}{\int_{0}^{1}{(1-\partial_x^2)\tv\tval\partial_x\hv}} \\
\label{e:aux-in-prooF1}
    &= \int_{t_0}^{t_1}{\int_{0}^{1}{\tv\tval\partial_x \hv}}
      -\int_{t_0}^{t_1}{\int_{0}^{1}{\partial_x^2\tv\tval\partial_x \hv}}.     
\end{align}
We bound the first term $\int_{t_0}^{t_1}{\int_{0}^{1}{\tv\tval \partial_x \hv}}$ of the right hand side of \eqref{e:aux-in-prooF1} by
\begin{align}
\notag
\left|\int_{t_0}^{t_1}{\int_{0}^{1}{\tv\tval\partial_x \hv}}\right| 
    &\leq \Vert\hv\Vert_{L^{\infty}([0,T],W^{1,\infty}(0,1))} \int_{t_0}^{t_1}{\left(\Vert\tv\Vert_{L^2(0,1)}\,\Vert\tval\Vert_{L^2(0,1)}\right)} \\
\label{e:aux-in-prooF2}
    &\leq C \int_{t_0}^{t_1}{\left(\Vert\tv\Vert_{H^1(0,1)}^2+|\partial_x \tv(t\cdot,0)|^2\right)v_l^+}.
\end{align}
For the second term $\int_{t_0}^{t_1}{\int_{0}^{1}{\partial_x^2\tv\tval\partial_x \hv}}$ of the right hand side of \eqref{e:aux-in-prooF1}, we have:
\begin{equation}\label{e:aux-in-prooF3}
\int_{t_0}^{t_1}{\int_{0}^{1}{\partial_x^2\tv\tval\,\partial_x \hv}}
    = \int_{t_0}^{t_1}{[\partial_x\tv\tval \partial_x\hv]_0^1}
     -\int_{t_0}^{t_1}{\int_{0}^{1}{\partial_x\tv\partial_x(\tval \partial_x\hv)}}.
\end{equation}
Once again, we bound the trilinear term by using 
$\Vert\tval\Vert_{H^2}=|\partial_x \tv(\cdot,0)|v_l^+\Vert u_l\Vert_{H^2}$.
\begin{equation}\label{e:aux-in-prooF4}
\left|\int_{t_0}^{t_1}{\int_{0}^{1}{\partial_x\tv\partial_x(\tval\partial_x\hv)}}\right|
    \leq  C \int_{t_0}^{t_1}{\left(\Vert\tv\Vert_{H^1(0,1)}^2+|\partial_x \tv(\cdot,0)|^2\right)v_l^+}.
\end{equation}
The boundary term $\int_{t_0}^{t_1}{(\partial_x\tv\tval \partial_x\hv)_{|x=1}}$ is equal to $0$ as $\tval_{|x=1}=0$  by definition. 
The term $\int_{t_0}^{t_1}{(\partial_x\tv\tval \partial_x\hv)_{|x=0}}$ can be bounded through
\begin{equation}\label{e:aux-in-prooF5}
\left|\int_{t_0}^{t_1}{(\partial_x\tv\tval \partial_x\hv)_{|x=1}}\right|
    \leq C\int_{t_0}^{t_1}{|\partial_x\tv(\cdot,0)|^2v_l^+},
\end{equation}
Combining \eqref{e:aux-in-prooF1}-\eqref{e:aux-in-prooF5}, we get
\begin{equation}\label{e:aux-in-pcomb4}
\left|\int_{t_0}^{t_1}{\int_{0}^{1}{\ty\hv\,\partial_x \tval}}\right|
    \leq C \int_{t_0}^{t_1}{\left(\Vert\tv\Vert_{H^1(0,1)}^2+|\partial_x \tv(\cdot,0)|^2\right)v_l^+}.
\end{equation}

\item
At last, we control the boundary term $\int_{t_0}^{t_1}{\tyl v_l \tval(\cdot,0)}$ through
\begin{equation}
\left|\int_{t_0}^{t_1}{\tyl v_l \tval(\cdot,0)}\right|
    \leq \tanh(1)\int_{t_0}^{t_1}{|\tyl|^2(v_l^+)^2} + \tanh(1)\int_{t_0}^{t_1}{|\partial_x \tv(\cdot,0)|^2v_l^+}
\end{equation}
\end{itemize}

Combining all the estimates \eqref{e:aux-in-pcomb0}, \eqref{e:aux-in-pcomb1}, \eqref{e:aux-in-pcomb2}, \eqref{e:aux-in-pcomb3}, \eqref{e:aux-in-pcomb4} for all the terms of \eqref{e:weakform-tva-test}, we get that there exists a constant $C>0$ independent of $t_0,t_1$ such that :
\begin{align} \notag
\int_{t_0}^{t_1}{v_l^+\frac{d}{dt}\left(|\partial_x\tv(\cdot,0)|^2\right)}
    \leq & C \int_{t_0}^{t_1}{\left(\Vert\tv\Vert_{H^1(0,1)}^2+|\partial_x \tv(\cdot,0)|^2\right)v_l^+} \\ \label{e:aux-in-final}
     &+ \frac{1}{2}\int_{t_0}^{t_1}{|\partial_x\tv(\cdot,1)|^2|v_r|v_l^+}.
     + \int_{t_0}^{t_1}{|\tyl|^2(v_l^+)^2}.
\end{align}
Using Lemma \ref{l:tva-time-reg}, we get that $|\partial_x\tv(\cdot,0)|^2\in W^{1,\infty}(0,T)$.
Therefore the function 
\begin{equation*}
U:t\mapsto \int_{0}^{t}{v_l^+(s)\frac{d}{dt}\left(|\partial_x\tv(s,0)|^2\right)\mathrm{d}s},
\end{equation*}
is also $W^{1,\infty}$ and its derivative in the weak sense is equal to 
\begin{equation*}
t\mapsto v_l^+(t)\frac{d}{dt}\left(|\partial_x\tv(t,0)|^2\right),
\end{equation*}
which is in $L^{\infty}$.
By Rademacher theorem, $U$ is differentiable in the classical sense for almost every $t\in ]0,T[$.
For such a fix $t$, and for $\varepsilon>0$, we take $t_0=t-\varepsilon$ and $t_1=t+\varepsilon$ in \eqref{e:aux-in-final}.
When $\varepsilon$ goes to zero, every term converges, and we get
\begin{equation}\label{e:aux-in-final2}
v_l^+\frac{d}{dt}\left(|\partial_x\tv(\cdot,0)|^2\right)
    \leq  C \left(\Vert\tv\Vert_{H^1(0,1)}^2+|\partial_x \tv(\cdot,0)|^2\right)v_l^+
     + \frac{1}{2}|\partial_x\tv(\cdot,1)|^2|v_r|v_l^+
     +|\tyl|^2(v_l^+)^2.
\end{equation}
We can divide by $v_l^+$ whenever we are in $\Gamma_l$, which gives the inequality \eqref{i:aux-ineq-l}, as wanted.

The proof of inequality \eqref{i:aux-ineq-r} is similar at each step, except we use the test function
\begin{equation}
\tvar :=  \partial_x\tv (t,1)\, v_r^-(t)\, u_r(x)
\end{equation}
instead of $\tval$.
\end{proof}

\begin{rem}
If we are ready to increase the constant $C$ in front of $\left(\Vert\tv(t,\cdot)\Vert_{H^1(0,1)}^2+|\partial_x \tv(t,1)|^2\right)$ in \eqref{i:aux-ineq-l}, then we could change the constants in front of $|\partial_x \tv(t,1)|^2|v_r(t)|$ and $|\tyl|^2v_l^+$ in this inequality.
\end{rem}


\subsection{Gronwall argument and end of the proof}\label{s:gron}

We define the functions $E$, $E_l$ and $E_r$ by:
\begin{equation}
E(t)  := \Vert \tv(t,\cdot)\Vert_{H^1(0,1)}^2,  \quad
E_l(t):= |\partial_x \tv(t,0)|^2,               \quad
E_r(t):= |\partial_x \tv(t,1)|^2.
\end{equation}
By Lemmata \ref{l:tv-time-reg} and \ref{l:tva-time-reg}, we know that $E$ is well-defined and Lipschitz.
Moreover, in the case where the boundary conditions for $y^1$ and $y^2$ are the same,
\begin{equation}
y_l^{1,c}=y_l^{2,c}
 \quad \text{and} \quad 
y_r^{1,c}=y_r^{2,c},
\end{equation} 
we can combine \eqref{i:energ-ineq}, \eqref{i:aux-ineq-l} and \eqref{i:aux-ineq-r}, to get
\begin{subequations}
\label{i:gron-syst}
\begin{align}
E'+E_l'+E_r'                       &\leq C(E+E_l+E_r)       &\text{ on } \Gamma_l\cap\Gamma_r,      \\
E'+E_l' + \tfrac{1}{2}E_r|v_r|      &\leq C(E+E_l)           &\text{ on } \Gamma_l\setminus(\Gamma_l\cap\Gamma_r), \\
E'+E_r' + \tfrac{1}{2}E_l|v_l|      &\leq C(E+E_r)           &\text{ on } \Gamma_r\setminus(\Gamma_l\cap\Gamma_r), \\
E' +\tfrac{1}{2}E_l|v_l|+\tfrac{1}{2}E_r|v_r|   &\leq CE      &\text{ on } [0,T]\setminus (\Gamma_l\cup\Gamma_r). 
\end{align}
\end{subequations}
Therefore, we can use the Gronwall inequality to get uniqueness on each time interval where neither $v_l$ nor $v_r$ changes sign.
On such an interval $I=[T_0,T_1]$, one gets:
\begin{subequations}
\label{i:gron-syst-cons}
\begin{align}
(E+E_l+E_r)(T_1) &\leq \exp(C(T_1-T_0))\, (E+E_l+E_r)(T_0) &\text{ if } I\subset\Gamma_l\cap\Gamma_r, \\
(E+E_l)(T_1)     &\leq \exp(C(T_1-T_0))\, (E+E_l)(T_0)     &\text{ if } I\subset\Gamma_l\setminus(\Gamma_l\cap\Gamma_r), \\
(E+E_r)(T_1)     &\leq \exp(C(T_1-T_0))\, (E+E_r)(T_0)     &\text{ if } I\subset\Gamma_r\setminus(\Gamma_l\cap\Gamma_r), \\
E(T_1)           &\leq \exp(C(T_1-T_0))\, E(T_0)           &\text{ if } I\subset[0,T]\setminus (\Gamma_l\cup\Gamma_r). 
\end{align}
\end{subequations}
This implies that 
\begin{equation} \label{e:est-ener-final}
E(T_1) \leq \exp(C(T_1-T_0))\, (E+E_l+E_r)(T_0).
\end{equation}
This concludes the proof of the first part of Theorem \ref{t:main}.

Now let us assume that $y_0^1=y_0^2$.
Let us denote by $T_0<T_1<...<T_n<T$ the times where $v_l$ or $v_r$ change sign.
On $[0,T_0]$, and on each interval $[T_i,T_{i+1}]$ one has the estimate \eqref{e:est-ener-final}.
By induction, we obtain that $v$ is equal to zero for each $T_i$ and on $[0,T]$.

\begin{rem} \label{rem:granw-contrex}
Gronwall argument normally comes with stability estimates.
However, in our case, if initial data are non-zero, they could degenerate.

For example, take $T=1$ and $v_r$ and $v_l$ given by $v_l(t)=-1+t$ and $v_r(t)=1$.
For the sake of simplicity, we assume that $C=1$ here.
Then the functions $E$, $E_l$ and $E_r$ defined by $E(t):=\frac{e^t+1}{2}$, $E_l(t):=\frac{1}{2(1-t)}$ and $E_r(t):=0$ verify the system \eqref{i:gron-syst} on $[0,1]$.
But $E_l$ is going to infinity so we cannot continue estimates on $E$ after $t=1$.

This phenomenon cannot happen in the case of an initial data equal to zero, because in this case, $E_l=E_r=0$ for all $t$. 
\end{rem}

\begin{rem}
The aforementioned constant $C$ do depends on $\Vert y^1\Vert_{L^{\infty}}$ and $\Vert y^2\Vert_{L^{\infty}}$.
Due to this, our estimates cannot be used to prove the existence or uniqueness of a lower class of regularity than the one we use.
\end{rem}

\begin{rem}
If the boundary conditions $(y_l^{1,c},y_r^{1,c})$ and $(y_l^{2,c},y_r^{2,c})$ are not the same, one still get an \textit{a priori} estimate.
However, due to remark \ref{rem:granw-contrex}, one can see that this estimates no longer provides uniqueness in the cases where $v_l$ or $v_r$ changes sign.

A question that is still open, is to determine whether or not one could still get estimates if the two solutions we are comparing do not have the same boundary fluxes $v_l$ and $v_r$.
\end{rem}

\section{Proof of Theorem \ref{t:third}} \label{s:higher-order}

\subsection{Energy estimate} \label{s:energ-est-gen}

Let us take two weak solutions $(y^1,y_r^1,y_l^1)$ and $(y^2,y_r^2,y_l^2)$ of the transport-elliptic system associated with $\An$ with initial and boundary conditions $(y_0^1,\mathbf{v_l},\mathbf{v_r},y_r^{1,c},y_l^{1,c})$ and $(y_0^2,\mathbf{v_l},\mathbf{v_r},y_r^{2,c},y_l^{2,c})$.
We define the following functions 
\begin{align}
\ty   &:= y^1-y^2,                  &\tv  := v^1-v^2,             \\
\hy   &:= \frac{y^1+y^2}{2},        &\hv  := \frac{v^1+v^2}{2},   \\
\tyl  &:= y_l^1-y_l^2,              &\tyr := y_r^1-y_r^2,            
\end{align}
where the functions $v^1$ and $v^2$ are given through \eqref{e:syst-v-y-ell-gen}.
Let us remark here that we have $(v^1-v^2)_{|x=0}=(v^1-v^2)_{|x=1}=0$, ... , $\partial_x^{n-1}(v^1-v^2)_{|x=0}=\partial_x^{n-1}(v^1-v^2)_{|x=1}=0$ and $\left(\left(\frac{v^1+v^2}{2}\right)_{|x=0},...,\partial_x^{n-1}\left(\frac{v^1+v^2}{2}\right)_{|x=0}\right)=\mathbf{v_l}$ as well as $\left(\left(\frac{v^1+v^2}{2}\right)_{|x=1},...,\partial_x^{n-1}\left(\frac{v^1+v^2}{2}\right)_{|x=1}\right)=\mathbf{v_r}$. \\

We take the difference of Equation \eqref{e:weakform-CH_gen} for the solutions $y^1$ and $y^2$.
The function $\ty$ verifies the following equality for all $0\leq t_0\leq t_1 <T$ and for all test function $\phi\in H^1([t_0,t_1]\times[0,1])$:
\begin{align}
\notag
\int_{t_0}^{t_1}&{\int_{0}^{1}{\big(\ty\,\partial_t \phi+ (\ty\hv+\hy\tv)\,\partial_x \phi - (\ty\partial_x \hv+\hy\partial_x\tv)\tv\big)}}        \\
\label{e:weakform-CH-diff_gen}
    &= \int_{t_0}^{t_1}{\tyr v_r \phi(\cdot,1)} - \int_{t_0}^{t_1}{\tyl v_l \phi(\cdot,0)}
    +\int_{0}^{1}{\phi(t_1,\cdot)\ty(t_1,\cdot)}-\int_{0}^{1}{\phi(t_0,\cdot)\ty(t_0,\cdot)}.
\end{align}

The following Lemma is the generalization of Lemma \ref{l:tv-time-reg}, and its proof is similar.

\begin{lem}\label{l:gen-tv-time-reg}
The functions $\hv$ and $\tv$ lie in $L^{\infty}([0,T],W^{2n,\infty}(0,1))$.
Moreover the function $\tv$ lies in $W^{1,\infty}([0,T],H^n([0,1]))$.
\end{lem}

\begin{proof}
For the regularity in space of $\tv$ and $\hv$, write
\begin{equation}
(-1)^n\partial_x^{2n}\tv =-\sum_{k=0}^{n-1}{(-\partial_x^2)^k\tv}+\ty.
\end{equation}
We can conclude using Lemma \ref{l:tech-ODE}.

Let us now prove the regularity in time of $\tv$.
Let us fix two times $t_0<t_1$, and denote 
$$a_{t_0}^{t_1}(x):=\tv(t_1,x)-\tv(t_0,x).$$
Recalling that $\tv$ verifies \eqref{e:syst-v-y-ell-gen} with homogeneous boundary conditions, we obtain that for every function $g\in H^n_0(0,1)$
\begin{equation}
\int_{0}^{1}{(\ty(t_1,\cdot)-\ty(t_0,\cdot))g}
    = \int_{0}^{1}{\And a_{t_0}^{t_1}\cdot \And g}
\end{equation}
We apply it with $a_{t_0}^{t_1}$ instead of $g$
\begin{equation}
\int_{0}^{1}{(\ty(t_1,\cdot)-\ty(t_0,\cdot))a_{t_0}^{t_1}}
    = \int_{0}^{1}{|\And a_{t_0}^{t_1}|^2}
\end{equation}
Hence, we get the inequality
\begin{equation}\label{e:gen-a-H1norm}
\Vert a_{t_0}^{t_1}\Vert_{H^n(0,1)}^2
    \leq  \left\vert\int_{0}^{1}{(\ty(t_1,\cdot)-\ty(t_0,\cdot))a_{t_0}^{t_1}}\right\vert.
\end{equation}
Using \eqref{e:weakform-CH_gen} with $a_{t_0}^{t_1}$ instead of $\phi$ (considered as a function constant in time), we obtain that:
\begin{equation}\label{e:gen-a-test-func}
\int_{0}^{1}{(\ty(t_1,\cdot)-\ty(t_0,\cdot))a_{t_0}^{t_1}}
    = \int_{t_0}^{t_1}{\int_{0}^{1}{\big((\ty\hv+\hy\tv)\partial_x a_{t_0}^{t_1}
     - (\ty\partial_x \hv+\hy\partial_x\tv)a_{t_0}^{t_1} \big)}}.
\end{equation}
Combining \eqref{e:gen-a-H1norm} and \eqref{e:gen-a-test-func}, using that $n\geq 1$, we get that:
\begin{equation*}
\Vert a_{t_0}^{t_1}\Vert_{H^n(0,1)}
    \leq |t_1-t_0|\big(\Vert\ty\Vert_{L^{\infty}([0,T]\times[0,1])}\Vert\hv\Vert_{L^{\infty}([0,T],H^1(0,1))}
    +\Vert\hy\Vert_{L^{\infty}([0,T]\times[0,1])}\Vert\tv\Vert_{L^{\infty}([0,T],H^1(0,1))}\big)
\end{equation*}
Recalling that $a_{t_0}^{t_1}(x)=\tv(t_1,x)-\tv(t_0,x)$, we conclude that $\tv\in W^{1,\infty}([0,T],H^n([0,1]))$.
\end{proof}

\begin{rem}
Lemma \ref{l:gen-tv-time-reg} expresses the fact that $\partial_t \An \tv = \partial_x(\tv\hy+\hv\ty) $, which is in $L^{\infty}_tH^{-1}_x$.
By elliptic regularity, we could obtain a higher regularity for $\partial_t\tv$, but it is not needed here.
\end{rem}

\begin{prop}
There exists a constant $C>0$ such that the following inequality holds for almost every $t\in [0,T]$
\begin{equation} \label{e:energ-ineq-gen}
\frac{d}{dt}\left(\Vert \tv\Vert_{H^k(0,1)}^2\right) 
    + |\partial_x^n\tv(\cdot,1)|^2v_r 
    - |\partial_x^n\tv(\cdot,0)|^2v_l
    \leq C\Vert \hv\Vert_{W^{2n,\infty}(0,1)}\Vert \tv\Vert_{H^k(0,1)}^2
\end{equation}
\end{prop}

\begin{proof}
We take $\tv$ as a test function in \eqref{e:weakform-CH-diff_gen}, which gives
\begin{align} \notag
\int_{t_0}^{t_1}&{\int_{0}^{1}{\big(\ty\,\partial_t \tv+ (\ty\hv+\hy\tv)\,\partial_x \tv - (\ty\partial_x\hv+\hy\partial_x\tv)\tv \big)}}   \\ \label{e:gen-energ-deb}
    &= \int_{0}^{1}{\tv(t_1,\cdot)\ty(t_1,\cdot)}-\int_{0}^{1}{\tv(t_0,\cdot)\ty(t_0,\cdot)}.
\end{align}
Then, we simplify each term.

\begin{itemize}
\item 
The term $\int_{t_0}^{t_1}{\int_{0}^{1}{\ty\partial_t \tv}}$ can be treated as follows
\begin{align}
\notag
\int_{t_0}^{t_1}{\int_{0}^{1}{\ty\partial_t \tv}}
    &= \int_{t_0}^{t_1}{\int_{0}^{1}{\An \tv \partial_t \tv}} \\ \notag
    &= \int_{t_0}^{t_1}{\int_{0}^{1}{\And \tv\cdot \partial_t\And \tv}} \\ \label{e:gen-ty-partialttv}
    &= \tfrac{1}{2}\left[\Vert \And\tv\Vert_{L^2}^2\right]_{t_0}^{t_1}.
\end{align}

\item
Similarly, we get
\begin{equation}\label{e:gen-ty-tv}
\int_{0}^{1}{\ty\tv}=\Vert \And\tv\Vert_{L^2}^2.
\end{equation}

\item
The trilinear term $\int_{0}^{1}{\hy\tv \partial_x\tv}$ cancels with $\int_{0}^{1}{\hy\tv \partial_x\tv}$.

\item
To simplify the trilinear term $\int_{0}^{1}{\ty\hv \partial_x\tv}$, we first use Lemma \ref{l:IPP-itérée}
\begin{align}
\notag
\int_{0}^{1}{\ty\hv \partial_x\tv}
    &= \int_{0}^{1}{\An \tv \hv\partial_x\tv} \\ \label{e:gen-ty-hv-partialxtv1}
    &= \int_{0}^{1}{\And \tv \cdot \And(\hv\partial_x \tv)}
       -\left[\hv|\partial_x^n\tv|^2\right]_{0}^{1}.
\end{align}
Then we put all the derivatives on $\tv$, which can be done in a nice way due to Lemma~\ref{l:comutation-Ln-fct}
\begin{equation}\label{e:gen-ty-hv-partialxtv2}
\int_{0}^{1}{\And \tv \cdot \And(\hv\partial_x\tv)}
   = \int_{0}^{1}{\And \tv \cdot \hv \partial_x\And\tv}
     + \int_{0}^{1}{\And \tv \cdot \left[\hv,\And\right]\partial_x\tv}.
\end{equation}
Finally, we integrate by part to once again put all derivatives on $\tv$
\begin{align}
\notag
\int_{0}^{1}{\And \tv\cdot \hv \partial_x\And\tv}
    &= \tfrac{1}{2}\int_{0}^{1}{\hv \partial_x\left|\And\tv\right|^2} \\ \label{e:gen-ty-hv-partialxtv3}
    &= -\tfrac{1}{2}\int_{0}^{1}{\left|\And\tv\right|^2\partial_x\hv} 
      + \tfrac{1}{2}\left[\hv|\partial_x^n\tv|^2\right]_{0}^{1}.
\end{align}
Combining \eqref{e:gen-ty-hv-partialxtv1}, \eqref{e:gen-ty-hv-partialxtv2} and \eqref{e:gen-ty-hv-partialxtv3}, we get
\begin{equation}\label{e:gen-ty-hv-partialxtv4}
\int_{0}^{1}{\ty\hv \partial_x\tv}
    = - \tfrac{1}{2}\int_{0}^{1}{\left|\And\tv\right|^2\partial_x\hv}
     + \int_{0}^{1}{\And \tv \cdot \left[\hv,\And\right]\partial_x\tv}
     - \tfrac{1}{2}\left[\hv|\partial_x^n\tv|^2\right]_{0}^{1}.
\end{equation}

\item
To simplify the trilinear term $\int_{0}^{1}{\ty\partial_x\hv\tv}$, we use Lemma \ref{l:IPP-itérée}
\begin{align}
\notag
\int_{0}^{1}{\ty \partial_x\hv\tv}
    &= \int_{0}^{1}{\An \tv \partial_x\hv\tv} \\ \label{e:gen-ty-partialxhv-tv1}
    &= \int_{0}^{1}{\And \tv \cdot \And(\partial_x\hv \tv)}.
\end{align}
\end{itemize}

We substitute \eqref{e:gen-ty-partialttv}, \eqref{e:gen-ty-tv}, \eqref{e:gen-ty-hv-partialxtv4} and \eqref{e:gen-ty-partialxhv-tv1} into \eqref{e:gen-energ-deb} to get
\begin{align} \notag
\left[\Vert \And\tv\Vert_{L^2}^2\right]_{t_0}^{t_1}
    &+ \int_{t_0}^{t_1}{\left[\hv|\partial_x^n\tv|^2\right]_{0}^{1}} \\ \label{e:gen-energ-fin}
    &+ \int_{t_0}^{t_1}{\int_{0}^{1}{\left(-2\And \tv \cdot \And(\partial_x\hv \tv)
    + \left|\And\tv\right|^2\partial_x\hv
    - 2\And \tv \cdot \left[\tv,\And\right]\partial_x\tv\right)}}
    = 0.
\end{align}

Using Lemma \ref{l:comutation-Ln-fct}, we get that for all $0<t_0<t_1<T$, one has
\begin{align} \notag
\left[\Vert\tv\Vert_{H^k}^2\right]_{t_0}^{t_1}
    + \int_{t_0}^{t_1}{|\partial_x^n\tv(\cdot,1)|^2v_r}
    &- \int_{t_0}^{t_1}{|\partial_x^n\tv(\cdot,0)|^2v_l} \\
    &\leq C \int_{t_0}^{t_1}{\int_{0}^{1}{\Vert \hv\Vert_{W^{2n,\infty}(0,1)}\Vert \tv\Vert_{H^k(0,1)}^2}}
\end{align}

\end{proof}

\subsection{Auxiliary estimate} \label{s:aux-est-gen}

In this paragraph, we chose $I= [T_0,T_1]\subset [0,T]$ an interval such that $v_l$ and $v_r$ do not change sign on $I$.
Without loss of generality, we assume that:
\begin{equation}
    \forall t\in I, v_l(t)>0 \text{ and } v_r(t)>0.
\end{equation}

We construct the auxiliary test function $\tvaln$ as the solution to the following elliptic problem
\begin{subequations}
\begin{align}
\label{e:ell-tvaln}
\An \tvaln &= 0, \\
\label{e:BC-tvaln-r}
\forall i\in[\![0,n-1]\!], \quad  \S_i(\tvaln) (\cdot,1) &=0, \\
\label{e:BC-tvaln-l}
\forall i\in[\![0,n-1]\!], \quad  \B_i(\tvaln) (\cdot,0) &= -\B_i(\tv) (\cdot,0),
\end{align}
\end{subequations}
where the operators $\B_i$ where defined in Appendix $A$ through \eqref{e:def-Bi}.

Let us introduce the space $H^{n}_{0,r}(0,1)$ as the closure of $C^{\infty}_c([0,1))$ for the $H^n$ norm
\begin{equation}
H^{n}_{0,r}(0,1) := \{g\in H^n(0,1);\forall i\in[\![0,n-1]\!], \partial_x^i g(1) = 0 \}.
\end{equation}
It is the natural space to define $\tvaln$ as it is a solution to a Zaremba problem (Dirichlet on one side and Neumann on the other).

\begin{lem}
The function $\tvaln$ exists and is unique in $L^{\infty}(I,H^{n}_{0,r}(0,1))$ as the solution of the following variational problem:
\begin{equation} \label{e:tvaln-varform-Hk0r}
\forall g\in H^{n}_{0,r}(0,1), \quad 
\int_{0}^{1}{\And \tvaln (t,\cdot) \cdot \And g} 
    = - \sum_{i=0}^{n-1}{\B_i(\tv)(t,0)\, \S_i(g)(0)}.
\end{equation}
Moreover, the function $\tvaln$ lies in $L^{\infty}(I,W^{2n,\infty}(0,1))$.
\end{lem}

\begin{proof}
Since $\tv$ belongs to $L^{\infty}(I,H^{2n}(0,1))$, for each $i$, $t\mapsto \B_i(\tv)(t,0)$ belongs to $L^{\infty}(I)$.
Hence by Lax-Milgram, $\tvaln$ belongs to $L^{\infty}(I,H^{n}_{0,r}(0,1))$ and is the unique solution of problem \eqref{e:tvaln-varform-Hk0r} in this space.

Using Lemma \ref{l:tech-ODE}, one gets that $\tvaln$ belongs to $L^{\infty}(I,W^{2n,\infty}(0,1))$.
\end{proof}

Let $g\in H^n(0,1)$ be a function.
Using Lemma \ref{l:IPP-itérée}, with $\tvaln$ instead of $f$ and $g$ instead of $g$ one has
\begin{equation}
\int_{0}^{1}{\An \tvaln\,g} 
    = \int_{0}^{1}{\And \tvaln\cdot\And g} 
    + \sum_{i=0}^{n-1}{[\B_i(\tvaln)\S_i(g)]_0^1}.
\end{equation}
Now due to \eqref{e:ell-tvaln}, one has $\int_{0}^{1}{\An \tvaln\,g}=0$ and due to \eqref{e:BC-tvaln-l}, one has 
\begin{equation}
\sum_{i=0}^{n-1}{\B_i(\tvaln)(\cdot,0)\S_i(g)(0)}
    =-\sum_{i=0}^{n-1}{\B_i(\tvaln)(\cdot,0)\S_i(g)(0)}.
\end{equation}
Hence
\begin{equation} \label{e:tvaln-varform-Hk}
\int_{0}^{1}{\And \tvaln\cdot\And g} 
    = - \sum_{i=0}^{n-1}{\B_i(\tv)(\cdot,0)\S_i(g)(0)} 
      - \sum_{i=0}^{n-1}{\B_i(\tvaln)(\cdot,1)\S_i(g)(1)}.
\end{equation}

In particular, for $g\in H^{n}_{0}(0,1)$, one has 
\begin{equation} \label{e:tvaln-varform-Hk0}
\int_{0}^{1}{\And \tvaln (t,\cdot) \cdot \And g} = 0.
\end{equation}
As this will be useful later, let us remark that for every $g\in H^{n}_{0,r}(0,1)$, one has
\begin{equation} \label{e:partialxtvaln-varform}
\int_{0}^{1}{\And \partial_x\tvaln (t,\cdot) \cdot \And g} 
    = \sum_{i=0}^{n-1}{\B_i(\partial_x\tv)(t,0)\, \S_i(g)(0)}.
\end{equation}

Similarly to the case of the classical Camassa-Holm equation, the introduction of this auxiliary test function is in sight of an auxiliary inequality.
The purpose of the auxiliary inequality \eqref{e:aux-ineq-gen} is to control the entering energy fluxes.

\begin{prop}\label{l:aux-ineq-gen}
For almost every $t\in I$, we have the inequality
\begin{equation} \label{e:aux-ineq-gen}
\frac{1}{2}\frac{d}{dt}\left(\Vert\tvaln\Vert_{H^n}^2\right)
    + |\partial_x^n\tv(t,0)|^2v_l(t)
    \leq C \left(\Vert\tvaln\Vert_{H^n}^2+ \Vert\tv\Vert_{H^n}^2\right)
    + \tfrac{1}{4}|\partial_x^n \tv(t,1)|^2|v_r|
    + |\tyl|^2.
\end{equation}
\end{prop}

\begin{rem}
If $v_r<0$, one can similarly introduce the function $\tvarn$ as
\begin{subequations}
\begin{align}
\label{e:ell-tvarn}
\An \tvaln &= 0, \\
\label{e:BC-tvarn-r}
\forall i\in[\![0,n-1]\!], \quad \B_i(\tvaln) (\cdot,1) &= -\B_i(\tv) (\cdot,1), \\
\label{e:BC-tvarn-l}
\forall i\in[\![0,n-1]\!], \quad  \S_i(\tvaln) (\cdot,0) &=0 .
\end{align}
\end{subequations}
and get the inequality
\begin{equation} \label{e:aux-ineq-gen-r}
\frac{1}{2}\frac{d}{dt}\left(\Vert\tvarn\Vert_{H^n}^2\right)
    + |\partial_x^n\tv(t,1)|^2v_r(t)
    \leq C \left(\Vert\tvarn\Vert_{H^n}^2+ \Vert\tv\Vert_{H^n}^2\right)
    + \tfrac{1}{4}|\partial_x^n \tv(t,0)|^2|v_l|
    + |\tyr|^2.
\end{equation}
\end{rem}

In order to prove Proposition \ref{l:aux-ineq-gen}, let us prove Lemma \ref{l:aux-reg-gen} and Proposition \ref{p:partialxntvaln1}.
Lemma \ref{l:aux-reg-gen} states that the auxiliary function $\tvaln$ is regular enough to be used as a test function in \eqref{e:weakform-CH-diff_gen}. 
Proposition \ref{p:partialxntvaln1} is an inequality similar to the classical Rellich estimate on the normal and tangential derivatives of harmonic functions, see for example \cite{Hormander-ineq}. 
We will use Proposition \ref{p:partialxntvaln1} to control one of the boundary terms on the outgoing boundaries.

\begin{lem} \label{l:aux-reg-gen}
The function $\tvaln$ lies in $W^{1,\infty}(I,H^n(0,1))$.
\end{lem}

\begin{proof}
Let us prove the regularity in time of the function $\tvaln$.
We call $V$ the function
\begin{equation}
V:= \tv + \tvaln.
\end{equation}
Since $\tv$ already belongs to $W^{1,\infty}(I,H^{n}([0,1])$, it is sufficient to prove that $V$ belongs to that space as well.
Moreover, for all $g\in H^n_{0,r}$ and for every $t\in I$, one has
\begin{equation}\label{e:fatigue1}
\int_{0}^{1}{\And V(t,\cdot) \cdot \And g} 
    = \int_{0}^{1}{\ty(t,\cdot)\, g}.
\end{equation}
Moreover by taking $g\in H^n_{0,r}(0,1)$ as a test function in \eqref{e:weakform-CH-diff_gen}, one gets for every $T_0\leq t_0<t_1\leq T_1$
\begin{equation}\label{e:fatigue2}
\int_{0}^{1}{(\ty(t_1,\cdot)-\ty(t_0,\cdot))g}
    = \int_{t_0}^{t_1}{\int_{0}^{1}{\big((\ty\hv+\hy\tv)\partial_x g-(\ty\partial_x\hv+\hy\partial_x\tv)g\big)}}.
\end{equation}
We apply \eqref{e:fatigue1} and \eqref{e:fatigue2} with $V(t_1,\cdot)-V(t_0,\cdot)$ instead of $g$ and since $n\geq 1$ we get
\begin{equation}
\Vert V(t_1,\cdot)-V(t_0,\cdot)\Vert_{H^n} 
    \leq |t_1-t_0|(\Vert\hy\Vert_{L^{\infty}}\Vert \tv\Vert_{H^1}+\Vert\ty\Vert_{L^{\infty}}\Vert\hv\Vert_{H^1}).
\end{equation}
\end{proof}
\begin{prop} \label{p:partialxntvaln1}
There exists a constant $C>0$ such that for every $t\in I$
\begin{equation}
|\partial_x^n \tvaln(t,1)| 
    \leq C \Vert \tvaln(t,\cdot)\Vert_{H^n(0,1)}.
\end{equation}
\end{prop}

\begin{proof}
Let $\chi\in C^{\infty}(0,1)$ be a function equal to zero in a neighborhood of $0$ and equal to $1$ in a neighborhood of $1$.
We use Lemma \ref{l:Lambs-lemma} with $\tvaln$ instead of $f$ and of $g$, and $\chi$ instead of $w$.
\begin{align*} \notag
\int_{0}^{1}&{\left[\partial_x(\chi \cdot),\And\right](\tvaln) \cdot \And(\tvaln)}
     + \int_{0}^{1}{\And(\tvaln) \cdot \left[\chi\partial_x,\And\right](\tvaln)}  \\
   &= \left[\chi\And(\tvaln)\cdot\And(\tvaln)\right]_0^1 
     + \sum_{i=0}^{n-1}{[\B_i(\tvaln)S_i(\partial_x(\chi\tvaln))]_0^1}
     + \sum_{i=0}^{n-1}{[\B_i(\tvaln) S_i(\chi\partial_x\tvaln)]_0^1}.
\end{align*}
Using the assumptions on $\chi$, we get that
\begin{equation}
\left[\chi\And(\tvaln)\cdot\And(\tvaln)\right]_0^1
    = |\partial_x^n \tvaln(t,1)|^2,
\end{equation}
as well as
\begin{align}
\sum_{i=0}^{n-1}{[\B_i(\tvaln)S_i(\partial_x(\chi\tvaln))]_0^1}
    &= - |\partial_x^n \tvaln(t,1)|^2, \\
\sum_{i=0}^{n-1}{[\B_i(\tvaln) S_i(\chi\partial_x\tvaln)]_0^1}
    &= - |\partial_x^n \tvaln(t,1)|^2.
\end{align}
Therefore 
\begin{equation*}
|\partial_x^n \tvaln(t,1)|^2
    = - \int_{0}^{1}{\left[\partial_x(\chi \cdot),\And\right](\tvaln) \cdot \And(\tvaln)}
      - \int_{0}^{1}{\And(\tvaln) \cdot \left[\chi\partial_x,\And\right](\tvaln)},
\end{equation*}
which allows us to conclude that there exists a constant $C$ such that
\begin{equation}
|\partial_x^n \tvaln(t,1)|^2 
    \leq C 
          \Vert\chi\Vert_{W^{n+1,\infty}(0,1)}
          \Vert \tvaln(t,\cdot) \Vert_{H^n(0,1)}^2.
\end{equation}
\end{proof}
\begin{proof}[Proof of Proposition \ref{p:aux-ineq}]
We take $\tvaln$ as an auxiliary test function in \eqref{e:weakform-CH-diff_gen}, which gives
\begin{align} \notag
\int_{t_0}^{t_1}&{\int_{0}^{1}{\big(\ty\,\partial_t \tvaln + (\ty\hv+\hy\tv)\,\partial_x \tvaln -(\ty\partial_x\hv+\hy\partial_x\tv)\tvaln \big)}} 
    = \\ \label{e:gen-aux-deb}
    &-\int_{t_0}^{t_1}{\tvaln(\cdot,0)\ty(\cdot,0)} 
    +\int_{0}^{1}{\tvaln(t_1,\cdot)\ty(t_1,\cdot)}
    -\int_{0}^{1}{\tvaln(t_0,\cdot)\ty(t_0,\cdot)}.
\end{align}
Then we simplify each term.

\begin{itemize}
\item 
First let us fix $t\in I$ and compute $\int_{0}^{1}{\tvaln(t,\cdot)\ty(t,\cdot)}$
\begin{align}
\notag
\int_{0}^{1}{\ty(t,\cdot)\tvaln(t,\cdot)}
    =& \int_{0}^{1}{\An\tv(t,\cdot)\tvaln(t,\cdot)} \\ \notag
    =& \int_{0}^{1}{\And\tv(t,\cdot)\cdot \And\tvaln(t,\cdot)} 
      + \sum_{i=0}^{n-1}{[\B_i(\tv)(t,\cdot) \, \S_i(\tvaln)(t,\cdot)]_0^1} \\ \notag
    =& \sum_{i=0}^{n-1}{\B_i(\tv)(t,0) \, \S_i(\tvaln)(t,0)} \\ \notag  
    =& - \int_{0}^{1}{\left|\And \tvaln(t,\cdot)\right|^2}   \\ \label{e:trucmuch1}
    =& - \Vert\tvaln(t,\cdot)\Vert_{H^n}^2.
\end{align}
We used Lemma \ref{l:IPP-itérée} to get from the first line to the second. 
We used \eqref{e:tvaln-varform-Hk0} with $\tv$ instead of $g$ and the fact that $\tvaln$ belong to $H^k_{0,r}$ to get from the second to the third. 
Then we used \eqref{e:tvaln-varform-Hk0r} with $\tvaln$ instead of $g$ to get the last line.

\item
The same computation allows us to simplify $\int_{t_0}^{t_1}{\int_{0}^{1}{\ty \partial_t \tvaln}}$
\begin{align}
\notag
\int_{t_0}^{t_1}{\int_{0}^{1}{\ty \partial_t \tvaln}}
    =& \int_{t_0}^{t_1}{\int_{0}^{1}{\An\tv\, \partial_t \tvaln}} \\ \notag
    =& \int_{t_0}^{t_1}{\sum_{i=0}^{n-1}{\B_i(\tv)(t,0) \, \S_i(\partial_t\tvaln)(t,0)}} \\ \label{e:trucmuch2}
    =& - \int_{t_0}^{t_1}{\int_{0}^{1}{\And \tvaln \cdot \partial_t\And \tvaln}}.
\end{align}
We can simplify \eqref{e:trucmuch2} by integration by parts in time
\begin{align} \notag
\int_{t_0}^{t_1}{\int_{0}^{1}{\And \tvaln \cdot \partial_t\And \tvaln}}
    &= \tfrac{1}{2}\left(\int_{0}^{1}{\left|\And \tvaln(t_1,\cdot)\right|^2}-\int_{0}^{1}{\left|\And \tvaln(t_0,\cdot)\right|^2}\right) \\ \label{e:trucmuch3}
    &= \tfrac{1}{2}\Vert\tvaln(t_1,\cdot)\Vert_{H^n}^2
      - \tfrac{1}{2}\Vert\tvaln(t_0,\cdot)\Vert_{H^n}^2.
\end{align}

\item
The terms $\int_{0}^{1}{\hy \tv \partial_x \tvaln}$ and $\int_{0}^{1}{\hy \partial_x\tv \tvaln}$ can be bounded by the Cauchy-Schwarz inequality
\begin{align}
\left|\int_{0}^{1}{\hy \tv \partial_x \tvaln}\right|
    & \leq \Vert\hy\Vert_{L^{\infty}}(\Vert\tv\Vert_{H^1}^2+\Vert\tvaln\Vert_{H^1}^2), \\
\left|\int_{0}^{1}{\hy \tv \partial_x \tvaln}\right|
    & \leq \Vert\hy\Vert_{L^{\infty}}(\Vert\tv\Vert_{H^1}^2+\Vert\tvaln\Vert_{H^1}^2).
\end{align}

\item
The term $\int_{0}^{1}{\ty\partial_x \hv \tvaln}$ can be simplified using Lemma \ref{l:IPP-itérée}
\begin{align} \notag
\int_{0}^{1}{\ty\partial_x \hv \tvaln}
    &= \int_{0}^{1}{\An \tv \partial_x \hv \tvaln} \\ \notag
    &= \int_{0}^{1}{\And \tv \cdot\And(\partial_x \hv \tvaln)} 
      + \sum_{i=0}^{n-1}{\B_i(\tv)(0)S_i(\partial_x \hv \tvaln)} \\
    &= \int_{0}^{1}{\And (\tv-\tvaln) \cdot\And(\partial_x \hv \tvaln)}.
\end{align}
Then, using the Cauchy-Schwarz inequality once again as well as Lemma \ref{l:comutation-Ln-fct}
\begin{equation}
\left|\int_{0}^{1}{\ty\partial_x \hv \tvaln} \right|
    \leq \Vert\hv\Vert_{W^{n+1,\infty}}(\Vert\tv\Vert_{H^n}^2+\Vert\tvaln\Vert_{H^n}^2).
\end{equation}

\item
Let us simplify $\int_{0}^{1}{\ty \hv \partial_x \tvaln}$.
We apply Lemma \ref{l:Lambs-lemma} with $\tv$ instead of $f$, $\tvaln$ instead of $g$ and $\hv$  instead of $w$.
\begin{align} \notag
\int_{0}^{1}{\hv\ty \, \partial_x \tvaln}
    = -& \int_{0}^{1}{\left[\partial_x(\hv \cdot),\And\right](\tv) \cdot \And(\tvaln)}
      - \int_{0}^{1}{\And(\tv) \cdot \left[\hv\partial_x,\And\right](\tvaln)} \\ \label{e:machin-1}
      +& \sum_{i=0}^{n-1}{[\B_i(\tvaln)S_i(\partial_x(\hv\tv))]_0^1}
      + \sum_{i=0}^{n-1}{[\B_i(\tv) S_i(\hv\partial_x\tvaln)]_0^1}
      + \left[\hv\And(\tvaln)\cdot\And(\tv)\right]_0^1.
\end{align}
Both integrals can be bounded by using Lemma \ref{l:comutation-Ln-fct} as follows
\begin{align}
\left|\int_{0}^{1}{\left[\partial_x(\hv \cdot),\And\right](\tv) \cdot \And(\tvaln)}\right|
     &\leq C \Vert \hv\Vert_{W^{n+1,\infty}(0,1)}\left( \Vert\tv\Vert_{H^n(0,1)}^2+\Vert\tvaln\Vert_{H^n(0,1)}^2\right) \\
\left|\int_{0}^{1}{\And(\tv) \cdot \left[\hv\partial_x,\And\right](\tvaln)}\right|
     &\leq C \Vert \hv\Vert_{W^{n,\infty}(0,1)} \left(\Vert\tv\Vert_{H^n(0,1)}^2+\Vert\tvaln\Vert_{H^n(0,1)}^2\right).
\end{align}
For $i\leq n-2$, one has $\S_i(\partial_x(\hv\tv))=0$, and $\S_{n-1}(\partial_x(\hv\tv))=\hv \partial_x^n \tv $.
Moreover for any function $f$, $\B_{n-1}(f)=- \partial_x^n f$.
Therefore,
\begin{equation}
\sum_{i=0}^{n-1}{[\B_i(\tvaln)S_i(\partial_x(\hv\tv))]_0^1}
    = - \partial_x^n\tv(\cdot,1)\partial_x^n\tvaln(\cdot,1) v_r
     - |\partial_x^n\tv(\cdot,0)|^2 v_l.
\end{equation}
We also have
\begin{equation}\label{e:machin0}
\left[\hv\And(\tvaln)\cdot\And(\tv)\right]_0^1
    = \partial_x^n\tv(\cdot,1)\partial_x^n\tvaln(\cdot,1) v_r
     + |\partial_x^n\tv(\cdot,0)|^2 v_l.
\end{equation}
Moreover using the variational formulation \eqref{e:tvaln-varform-Hk} for $\tvaln$, one gets 
\begin{align}
\notag
\sum_{i=0}^{n-1}{[\B_i(\tv) S_i(\hv\partial_x\tvaln)]_0^1}
    =& \sum_{i=0}^{n-1}{\B_i(\tv)(\cdot,1) S_i(\hv\partial_x\tvaln)(\cdot,1)}
       - \sum_{i=0}^{n-1}{\B_i(\tv)(\cdot,0) S_i(\hv\partial_x\tvaln)(\cdot,0)} \\ \notag
    =& \B_{n-1}(\tv)(\cdot,1) \S_{n-1}(\hv\partial_x \tvaln)(\cdot,1)  \\ \notag
      &+ \int_{0}^{1}{\And \tvaln\cdot \And(\hv\partial_x\tvaln)} 
       + \sum_{i=0}^{n-1}{\B_i(\tvaln)(\cdot,1)\S_i(\hv\partial_x\tvaln)(1)} \\ \label{e:bidule1}
    =& -(\partial_x^n \tv(\cdot,1)+\partial_x^n \tvaln(\cdot,1))\partial_x^n \tvaln(\cdot,1) v_r
       + \int_{0}^{1}{\And \tvaln\cdot \And(\hv\partial_x\tvaln)}.
\end{align}
We exchange $\hv$ and $\And$ up to a commutator in $\int_{0}^{1}{\And \tvaln\cdot \And(\hv\partial_x\tvaln)}$
\begin{equation} \label{e:bidule2}
\int_{0}^{1}{\And \tvaln\cdot \And(\hv\partial_x\tvaln)}
    = \int_{0}^{1}{\And \tvaln\cdot \hv\partial_x\And\tvaln}
     +\int_{0}^{1}{\And \tvaln\cdot \left[\And,\hv\partial_x\right]\tvaln},
\end{equation}
then we integrate by part
\begin{align}
\notag
\int_{0}^{1}{\And \tvaln\cdot \hv\partial_x\And\tvaln}
    &= \tfrac{1}{2}\int_{0}^{1}{ \hv\partial_x\left(\left|\And\tvaln\right|^2\right)} \\ \label{e:bidule3}
    &= -\tfrac{1}{2}\int_{0}^{1}{\partial_x \hv \left|\And\tvaln\right|^2} 
     + \left[\left|\partial_x^n\tvaln\right|^2 \hv\right]_0^1.
\end{align}
Combining \eqref{e:bidule1}, \eqref{e:bidule2} and \eqref{e:bidule3}, we get that
\begin{align} \notag
\sum_{i=0}^{n-1}{[\B_i(\tv) S_i(\hv\partial_x\tvaln)]_0^1}
    =& - \partial_x^n \tv(\cdot,1)\partial_x^n \tvaln(\cdot,1) v_r
       + \int_{0}^{1}{\And \tvaln\cdot \left[\And,\hv\partial_x\right]\tvaln} \\ \label{e:machin1}
      &- \tfrac{1}{2}\int_{0}^{1}{\partial_x \hv |\And\tvaln|^2} 
       - |\partial_x^n \tv(\cdot,0)|^2 v_l.
\end{align}
Once again, we bound the trilinear term as follows
\begin{align}
\left|\int_{0}^{1}{\And \tvaln\cdot \left[\And,\hv\partial_x\right]\tvaln}\right|
    &\leq C \Vert \hv\Vert_{W^{n+1,\infty}} \Vert\tvaln\Vert_{H^n}^2, \\
\left|\int_{0}^{1}{\partial_x \hv \left|\And\tvaln\right|^2}\right|
    &\leq \Vert \hv\Vert_{W^{1,\infty}} \Vert\tvaln\Vert_{H^n}^2.
\end{align}
Using Proposition \ref{p:partialxntvaln1}, we can bound $ \partial_x^n \tv(\cdot,1)\partial_x^n \tvaln(\cdot,1) v_r$ as follows
\begin{align} \notag
\partial_x^n \tv(\cdot,1)\partial_x^n \tvaln(\cdot,1) v_r
    &\leq \tfrac{1}{2} \left(|\partial_x^n \tv(\cdot,1)|^2 v_r +  |\partial_x^n \tvaln(\cdot,1)|^2 v_r\right) \\ \label{e:machin2}
    &\leq \tfrac{1}{2} |\partial_x^n \tv(\cdot,1)|^2 v_r + C \Vert\tvaln\Vert_{H^n}^2.
\end{align}
Using \eqref{e:machin-1}-\eqref{e:machin0} and \eqref{e:machin1}-\eqref{e:machin2}, we get
\begin{align} \notag
\Big|\int_{t_0}^{t_1}{\int_{0}^{1}{\hv\ty \, \partial_x \tvaln}} 
     &+ \int_{t_0}^{t_1}{|\partial_x^n\tv(\cdot,0)|^2v_l} \Big|    \\ \label{e:trucmuch4}
    &\leq C \Vert\hv\Vert_{L^{\infty}([0,T],W^{n+1,\infty}(0,1))} \int_{t_0}^{t_1}{\left(\Vert\tv\Vert_{H^n}^2+\Vert\tvaln\Vert_{H^n}^2\right)}
     + \tfrac{1}{4}\int_{t_0}^{t_1}{ |\partial_x^n\tv (\cdot,1)|^2v_r}.
\end{align}

\item
Finally, we control the boundary term $\int_{t_0}^{t_1}{\tvaln(\cdot,0)\ty(\cdot,0)}$ using classical trace theorem
\begin{equation} \label{e:trucmuch5}
\left|\int_{t_0}^{t_1}{\tvaln(\cdot,0)\ty(\cdot,0)}\right|
    \leq \int_{t_0}^{t_1}{\left(\Vert\tvaln\Vert_{H^1}^2 + |\tyl|^2\right)}.
\end{equation}
\end{itemize}

Now, let us assemble \eqref{e:trucmuch1}, \eqref{e:trucmuch2}, \eqref{e:trucmuch3}, \eqref{e:trucmuch4} and \eqref{e:trucmuch5} to get
\begin{align} \notag
\tfrac{1}{2}\left[\Vert\tvaln\Vert_{H^n}^2\right]_{t_0}^{t_1}
     &+ \int_{t_0}^{t_1}{|\partial_x^n\tv(\cdot,0)|^2v_l} \\ \label{e:aux-ineq-gen-proof}
   \leq C & \int_{t_0}^{t_1}{\left(\Vert\tv\Vert_{H^n}^2+\Vert\tvaln\Vert_{H^n}^2\right)}
     + \tfrac{1}{4}\int_{t_0}^{t_1}{ |\partial_x^n\tv (\cdot,1)|^2v_r}
     + \int_{t_0}^{t_1}{|\tyl|^2}.
\end{align}
\end{proof}

\subsection{Gronwall argument and end of the proof}

Let $I_1,..., I_K$ be the intervals on which neither $v_l$ nor $v_l$ change sign.

We assume that the initial and boundary conditions are the same.
We prove by induction on $k$ that $\tv$ is equal to zero on $I_k$.
First for every $k\in[\![1,K]\!]$, we construct $\tvaln$ and/or $\tvarn$ on $I_k$ according to the signs of $v_l$ and $v_r$ on $I_k$.

\begin{itemize}
\item 
Initialization step: 
by hypothesis, $\tv$ is equal to $0$ at time zero.
\item
Induction step:
let us fix $k\in[\![1,K]\!]$ assume that $\tv$ is equal to $0$ at the beginning of $I_k$.
Then, the auxiliary functions created on interval $I_{k}$ are equals to $0$ at the beginning of interval~$I_{k}$.
We denote by $E_{rel,k}:I_{k}\rightarrow \R_+$ the quantity:
\begin{equation}
E_{rel,k}(t) 
    :=
    \begin{cases}
      \Vert\tv\Vert_{H^n}^2
         +\Vert\tvaln\Vert_{H^n}^2
         +\Vert\tvarn\Vert_{H^n}^2
         & \text{ if } I_{k}\subset \Gamma_l\cap\Gamma_r, \\
       \Vert\tv\Vert_{H^n}^2
         +\Vert\tvaln\Vert_{H^n}^2 
         & \text{ if }  I_{k}\subset\Gamma_l\setminus(\Gamma_l\cap\Gamma_r), \\
       \Vert\tv\Vert_{H^n}^2
         +\Vert\tvarn\Vert_{H^n}^2
         & \text{ if } I_{k}\subset \Gamma_r\setminus(\Gamma_l\cap\Gamma_r), \\
        \Vert\tv\Vert_{H^n}^2
         & \text{ if } I_{k}\subset [0,T]\setminus (\Gamma_l\cup\Gamma_r). 
    \end{cases}
\end{equation}
We sum inequality \eqref{e:energ-ineq-gen} with two times inequality \eqref{e:aux-ineq-gen} if $I_{k}\subset \Gamma_l$ and two times inequality \eqref{e:aux-ineq-gen-r} if $I_{k}\subset \Gamma_r$.
One gets that there exists a constant $C>0$ such that
\begin{equation}
E_{rel,k}'(t) 
    + \tfrac{1}{2}\left(|\partial_x^n\tv(t,0)|^2|v_l|
                        +|\partial_x^n\tv(t,1)|^2|v_r|\right)
    \leq C E_{rel,k}(t).
\end{equation}
Hence, by the Gronwall inequality, since $E_{rel,k}$ is equal to zero at the beginning of $I_{k}$, it is equal to zero on $I_{k}$.
In particular, $\tv$ is equal to $0$ on $I_k$.
Since $\tv$ belongs to $C^0([0,T],H^n(0,1))$, we get that $\tv$ is equal to $0$ at the beginning of $I_{k+1}$, which concludes the induction as well as the proof of Theorem \ref{t:third}.
\end{itemize}

\appendix

\section{Integration by parts and commutator for $\An$}

\begin{lem}\label{l:IPP-itérée}
Let $g\in \mathrm{H}^{2n}(0,1)$ and $g\in \mathrm{H}^n(0,1)$ be two functions.
We have the equality
\begin{equation} \label{e:IPP-itéré}
\int_{0}^{1}{\An f\,g} 
    = \int_{0}^{1}{\And f\cdot\And g} 
    + \sum_{i=0}^{n-1}{[\B_i(f)\S_i(g)]_0^1},
\end{equation}
where $\cdot$ is the standard scalar product on $\R^{n+1}$ and the operator $\B_i$ and $\S_i$ are defined through
\begin{subequations}
\begin{align} \label{e:def-Bi}
\forall x\in\{0,1\}, \quad 
\B_i (f)(x) &:= \sum_{k=i+1}^{n}{(-1)^{k+i}\partial_x^{2k-1-i}f(x)}, \\
\forall x\in\{0,1\}, \quad
\S_i (g)(x) &:= \partial_x^i g(x).
\end{align}
\end{subequations}
Let us remark that the operators $\B_i$ and $\S_i$ are boundary operators of respective order $2n-1-i$ and $i$.
\end{lem}

\begin{proof}
By induction on $k\in\N$ 
\begin{equation*}
 \forall f\in H^{2k}(0,1), \forall g\in H^{k}(0,1), \quad 
\int_{0}^{1}{(\partial_x^{2k}f)g} 
    = (-1)^k\int_{0}^{1}{(\partial_x^k f)(\partial_x^k g)} 
    + \sum_{i=0}^{k-1}{(-1)^i[(\partial_x^{2k-1-i}f)(\partial_x^{i}g)]_0^1}.
\end{equation*}
By summation on $k\in\{1,\dots,n\}$, we have
\begin{equation} \label{e:IPP-itéré1}
\int_{0}^{1}{\An f\,g} 
    = \int_{0}^{1}{\And f\cdot\And g} 
    + \sum_{k=0}^{n}{\sum_{i=0}^{k-1}{(-1)^{k+i}\left[\partial_x^{2k-1-i}f\partial_x^{i}g\right]_{0}^{1}}},
\end{equation}
which can be rewritten into \eqref{e:IPP-itéré}, since
\begin{align*}
\sum_{k=0}^{n}{\sum_{i=0}^{k-1}{(-1)^{k+i}\left[\partial_x^{2k-1-i}f\partial_x^{i}g\right]_{0}^{1}}}
    &= \sum_{i=0}^{n-1}{\sum_{k=i+1}^{n}{(-1)^{k+i}\left[\partial_x^{2k-1-i}f\partial_x^{i}g\right]_{0}^{1}}} \\
    &= \sum_{i=0}^{n-1}{[\B_i (f) \S_i(g)]_0^1}.
\end{align*}
\end{proof}

Let $k\in\N^*$, $f\in W^{k,\infty}(0,1)$, and $\mathcal{A}$ a differential operator of order $k$.
We denote by $[\mathcal{A},f]$ the commutator operator
\begin{equation}
\forall g\in H^{k}(0,1), \quad [f,\mathcal{A}]g := f\mathcal{A} g - \mathcal{A}(fg). 
\end{equation}
As before, there will be trilinear term to simplify in our estimates.
We bound them using the two following Lemmata.
Lemma \ref{l:comutation-Ln-fct} is a simple consequence of Leibniz formula.
Lemma \ref{l:Lambs-lemma} is the consequence of a repeated use of Lemma \ref{l:IPP-itérée}.

\begin{lem}\label{l:comutation-Ln-fct}
There exists a constant $C>0$ depending only on $k$ and $\mathcal{A}$ such that
\begin{equation}
\forall g\in H^{k}(0,1), \quad
\Vert[f,\mathcal{A}]g\Vert_{L^2} \leq C \Vert f\Vert_{W^{k,\infty}}\Vert g\Vert_{H^{k-1}}. 
\end{equation}
\end{lem}
\begin{lem}\label{l:Lambs-lemma}
Let $f,g,w\in H^{2n}(0,1)$ be functions.
We have the following equality
\begin{align} \notag
\int_{0}^{1}{w\An (f) \, \partial_x g}
    &+ \int_{0}^{1}{\partial_x(wf) \, \An(g)}
     + \int_{0}^{1}{\left[\partial_x(w \cdot),\And\right](f) \cdot \And(g)}
     - \int_{0}^{1}{\And(f) \cdot \left[\And, w\partial_x\right](g)} \\ \label{e:lambs-equality}
   &= \sum_{i=0}^{n-1}{[\B_i(g)S_i(\partial_x(wf))]_0^1}
     + \sum_{i=0}^{n-1}{[\B_i(f) S_i(w\partial_xg)]_0^1}
     + \left[w\And(f)\cdot\And(g)\right]_0^1.
\end{align}
\end{lem}

\begin{proof}
Let us start by applying Lemma \ref{l:IPP-itérée} with $f$ and $w\partial_x g$ instead of $f$ and $g$
\begin{equation}\label{e:lambs-proof1}
\int_{0}^{1}{\An (f) \, w\partial_x g}
    = \int_{0}^{1}{\And (f) \cdot \And(w\partial_x g)} 
     + \sum_{i=0}^{n-1}{[\B_i(f)\S_i(w\partial_x g)]_0^1}.
\end{equation}
We exchange $\And$ and $w\partial_x$ up to a commutator:
\begin{equation}
\int_{0}^{1}{\And (f) \cdot \And(w\partial_x g)}
    = \int_{0}^{1}{\And (f) \cdot w\partial_x\And(g)}
     + \int_{0}^{1}{\And (f) \cdot \left[\And, w\partial_x\right] (g)}.
\end{equation}
We perform and integration by parts
\begin{equation}
\int_{0}^{1}{\And (f) \cdot w\partial_x\And(g)}
    = - \int_{0}^{1}{\partial_x(w\And (f)) \cdot\And(g)}
    + \left[w\And(f)\cdot\And(g)\right]_0^1.
\end{equation}
We exchange $\partial_x(w\cdot)$ and $\And$ up to another commutator
\begin{equation}
\int_{0}^{1}{\partial_x(w\And (f)) \cdot\And(g)}
    = \int_{0}^{1}{\And (\partial_x(wf)) \cdot \And(g)}
    + \int_{0}^{1}{\left[\partial_x(w\cdot), \And\right] (f) \cdot \And(g)}.
\end{equation}
Finally we apply Lemma \ref{l:IPP-itérée} once again, this time with $g$ and $\partial_x(wf)$ instead of $f$ and $g$
\begin{equation}\label{e:lambs-proof5}
\int_{0}^{1}{\And (\partial_x(wf)) \cdot \And(g)}
    = \int_{0}^{1}{\partial_x(wf) \, \An(g)}
     - \sum_{i=0}^{n-1}{\B_i(g) \S_i(\partial_x(wf))}.
\end{equation}
Combining \eqref{e:lambs-proof1}-\eqref{e:lambs-proof5}, we obtain \eqref{e:lambs-equality}.
\end{proof}

\section{Sketch of the proof of Theorem \ref{t:sec}} \label{s:app-exist}

Let $v$ be a function on $\Omega_T:= [0,T]\times [0,1]$, which verifies the boundary conditions \eqref{e:v-y-ell-gen-BCl} and \eqref{e:v-y-ell-gen-BCr}.
We denote by $\phi$ the flow of $v$.
It is defined as the unique solution of the following ODE
\begin{subequations}
\begin{align}
\partial_1 \phi(s,t,x) &= v(s,\phi(s,t,x)), \\
\phi(t,t,x) &= x.
\end{align}
\end{subequations}
The quantity $\phi(s,t,x)$ is the position at time $s$ of the particle which was in $x$ at time $t$.
The quantity $\phi(\cdot,t,x)$ is defined on an interval of time $[e(t,x),h(t,x)]$ where $e(t,x)$ and $h(t,x)$ are the time of entrance and exit of the domain for the particle going through $x$ at time $t$.

We define the sets $\Omega_L$, $\Omega_R$, $\Omega_I$ and $\Omega_{S}$ as:
\begin{align*}
\Omega_{S} :=& \{ (t,x)\in \Omega_T;\, 
          \exists s\in [e(t,x),h(t,x)], (\phi(s,t,x) =0 \text{ and } v_l(s) =0)  \\
                  &\text{ or } (\phi(s,t,x) =1 \text{ and } v_r(s) =0) \}  \\
               &\cup \{(s,\phi(s,0,0)); s\in [0,h(0,0)]\}
                   \cup \{(s,\phi(s,0,1)); s\in [0,h(0,1)]\}, \\
\Omega_I :=& \{ (t,x)\in \Omega_T\setminus \Omega_S;\, e(t,x)=0 \}, \\
\Omega_L :=& \{ (t,x)\in \Omega_T;\, i(t,x) > 0 \text{ and } \phi(e(t,x),t,x) = 0  \}, \\
\Omega_R :=& \{ (t,x)\in \Omega_T;\, i(t,x) > 0 \text{ and } \phi(e(t,x),t,x) = 1  \}.
\end{align*}
The sets $\Omega_I$, $\Omega_L$ and $\Omega_R$ are the sets of position of particles which enter the domain at time $0$, from the left and from the right respectively.
The set $\Omega_{S}$ is called the singular set, it contains the sets of particles which where at times $0$ at the boundary as well as the particles which where on the boundary with velocity zero at some point in time.

We define the function $y \in L^{\infty}(\Omega_T)$ by
\begin{itemize}
\item for $(x,t)\in \Omega_I$, $y(t,x) := y_0(\phi(0,t,x))\exp\left(-2\int_{0}^{t}{\partial_x v(s,\phi(s,t,x))\mathrm{ds}}\right)$,
\item for $(x,t)\in \Omega_L$, $y(t,x) := y_l^c(e(t,x))\exp\left(-2\int_{e(t,x)}^{t}{\partial_x v(s,\phi(s,t,x))\mathrm{ds}}\right)$,
\item for $(x,t)\in \Omega_R$, $y(t,x) := y_r^c(e(t,x))\exp\left(-2\int_{e(t,x)}^{t}{\partial_x v(s,\phi(s,t,x))\mathrm{ds}}\right)$.
\end{itemize}

We refer to \cite{Perrollaz} for the study of the transport equation with streching \eqref{e:trans-fort-strech-gen}.
The fact that we will use are :
\begin{itemize}
\item the function $y$ is well defined in $L^{\infty}(\Omega_T)$, together with the estimate
\begin{equation}\label{e:est-perro1}
\Vert y\Vert_{L^{\infty}(\Omega_T)} 
    \leq \max\left(\Vert y_0\Vert_{L^{\infty}},\Vert y_l^c\Vert_{L^{\infty}}, \Vert y_r^c\Vert_{L^{\infty}}\right) \exp\left(2T \Vert \partial_x v\Vert_{L^{\infty}}\right),
\end{equation}
\item the function $y$ is the unique solution of \eqref{e:trans-fort-strech-gen} with initial condition $y_0$ and boundary condition $y_r^c$ and $y_l^c$,
\item the function $y$ is in $W^{1,\infty}([0,T],H^{-1}(0,1))$ together with the estimate
\begin{equation}\label{e:est-perro2}
\Vert \partial_t y\Vert_{W^{1,\infty}([0,T],H^{-1}(0,1))} 
    \leq 3 \Vert y \Vert_{L^{\infty}(\Omega_T)} \Vert v\Vert_{L^{\infty}([0,T],W^{1,\infty}(0,1))},
\end{equation}
\end{itemize}
To simplify the notation, we denote $L^{\infty}_tW^{2n,\infty}_x$ instead of $L^{\infty}([0,T],W^{2n,\infty}(0,1))$ and similarly for $L^{\infty}_tW^{1,\infty}_x$, $W^{1,\infty}_tH^n_x$ as well as $W^{1,\infty}_t H^{-1}_x$.

We can then introduce the solution $u$ to the system
\begin{subequations}
\begin{align}
\An u        &= y,                                          \\ \label{e:bcl-u-ex-gen}
\mathbf{v_l} &= (\S_i(u)(0))_{i\in [\![0,n-1]\!]},          \\ \label{e:bcr-u-ex-gen}
\mathbf{v_r} &= (\S_i(u)(1))_{i\in [\![0,n-1]\!]}. 
\end{align}
\end{subequations}

We call $\mathcal{F}$ the operator which to $v\in L^{\infty}_tW^{2n,\infty}_x\cap W^{1,\infty}_tH^{n}_x$ associate $u\in L^{\infty}_tW^{2n,\infty}_x\cap W^{1,\infty}_tH^{n}_x$.

For $B_0$ and $B_1$ positive numbers, we introduce the space $C_{B_0,B_1,T}$ as 
\begin{equation}
C_{B_0,B_1,T} := \{ v\in   L^{\infty}_tW^{2n,\infty}_x\cap W^{1,\infty}_tH^{n}_x; \Vert v\Vert_{L^{\infty}W^{2n,\infty}}\leq B_0 \text{ and } \Vert v\Vert_{W^{1,\infty}H^{n}}\leq B_1 \}.
\end{equation}

The end of the proof is threefold : 
\begin{itemize}
\item find $B_0$ and $B_1$ such that $\mathcal{F}$ maps $C_{B_0,B_1,T}$ into itself,
\item prove that $C_{B_0,B_1,T}$ is compact with respect to $\Vert\cdot\Vert_{L^{\infty}_tW^{1,\infty}_x}$,
\item prove that $\mathcal{F}$ is continuous with respect to $\Vert\cdot\Vert_{L^{\infty}_tW^{1,\infty}_x}$.
\end{itemize}
Once all this is done one can conclude by applying Schauder's fixed point theorem.

\begin{lem}
There exists a time $T>0$ as well as $B_0$ and $B_1$ such that $\mathcal{F}$ maps $C_{B_0,B_1,T}$ into itself.
\end{lem}

\begin{proof}
Let us take $v\in L^{\infty}_tW^{2n,\infty}_x\cap W^{1,\infty}_tH^{n}_x$ and denote $u:= \mathcal{F}(v)$.
We denote by ,$c_1$ and $c_2$ 
\begin{align*}
c_1 &:= \max\left(\Vert y_0\Vert_{L^{\infty}_x},\Vert y_l^c\Vert_{L^{\infty}_t}, \Vert y_r^c\Vert_{L^{\infty}_t}\right), \\
c_2 &:= \Vert\mathbf{v_l}\Vert_{L^{\infty}_t}+\Vert\mathbf{v_r}\Vert_{L^{\infty}_t}
\end{align*}
the two constants depending on the initial and boundary data.
Combining the estimates \eqref{e:est-perro1} and \eqref{e:est-perro2} with the elliptic estimates from Lemma \ref{l:ell-reg-gen}, 
we get that there exists a constant $C$ depening only on $n$ such that
\begin{equation}
\Vert u\Vert_{L^{\infty}_tW^{2n,\infty}_x} 
    \leq C\left(c_1\exp(2T\Vert v\Vert_{L^{\infty}_tW^{2n,\infty}_x})+c_2\right),
\end{equation}
and 
\begin{equation}
\Vert \partial_t u\Vert_{L^{\infty}_tH^n_x} 
    \leq C\left(c_1\exp(2T\Vert v\Vert_{L^{\infty}_tW^{2n,\infty}_x})+c_2\right)\Vert v\Vert_{L^{\infty}_tW^{1,\infty}_x}.
\end{equation}
We chose $B_0:= 2C(c_1+c_2)$.
For $T$ small enough one has 
\begin{equation}
C\left(c_1\exp(2TB_0)+c_2\right)< 2C(c_1+c_2)=B_0,
\end{equation} 
we chose such a $T$.
Then we chose $B_1:=B_0^2$.
%
\end{proof}

\begin{lem}
For any $B_0$, $B_1$ and $T$, the space $C_{B_0,B_1,T}$ is compact with respect to the norm $\Vert\cdot\Vert_{L^{\infty}_tW^{1,\infty}_x}$.
\end{lem}

\begin{proof}
For $n=1$ this was done in \cite{Perrollaz}.

For $n\geq 2$, it is easier.
We have $W^{1,\infty}_tH^{n-1}_x \hookrightarrow W^{1,\infty}_tH^1_x$.
Therefore for $(t,x), (t',x')\in \Omega_T$ and $u\in C_{B_0,B_1,T}$, one has 
\begin{equation*}
|\partial_x u(t,x)-\partial_x u(t',x')| \leq |t-t'|\sqrt{|x-x'|}\Vert u\Vert_{W^{1,\infty}_tH^n_x},
\end{equation*}
and we conclude thanks to Ascoli's theorem.
\end{proof}

\begin{lem}
The operator $\mathcal{F}$ is continuous with respect to the norm $\Vert\cdot\Vert_{L^{\infty}_tW^{1,\infty}_x}$.
\end{lem}

The proof of this Lemma does not differ from Proposition 2.4 in \cite{Perrollaz}.

Combining all the arguments above, we proved the existence of $B_0$, $B_1$ and of a function $u\in C_{B_0,B_1,T}$, which is a fixed point of $\mathcal{F}$.
That is
\begin{itemize}
\item the unique solution $y$ of \eqref{e:trans-fort-strech-gen} with initial condition $y_0$ and boundary condition $y_r^c$ and $y_l^c$ is equal to $\An u$,
\item the function $u$ verifies the boundary condition \eqref{e:bcl-u-ex-gen}-\eqref{e:bcr-u-ex-gen}.
\end{itemize}
As is, we created a weak solution in the sense of distribution of the Camassa-Holm equation.
It is a weak solution in the sense of Definition \ref{d:weaksol-gen} due to the Theorem 3 in \cite{Boyer}.

\ \par \ 

\paragraph{\textbf{Acknowledgements.}}
The authors are partially supported by 
the Agence Nationale de la Recherche,
Project SINGFLOWS, 
ANR-18-CE40-0027-01.
The author warmly thank Franck Sueur for the careful reading and advises,
and David Lannes for interesting discussions on the subject.

\end{document}